\documentclass{amsart}
\usepackage{graphicx}
\usepackage{subfigure}

\usepackage{a4,pstricks,amsmath,latexsym,amssymb}

\newtheorem{thm}{Theorem}
\newtheorem{lem}[thm]{Lemma}
\newtheorem{prop}[thm]{Proposition}
\newtheorem{Def}[thm]{Definition}
\newtheorem{cor}[thm]{Corollary}
\newtheorem{Remark}[thm]{Remark}
\newtheorem{Question}{Question}

\newcommand{\cacher}[1]{}

\title[Schnyder decompositions for regular plane graphs]{Schnyder decompositions for regular plane graphs and application to drawing}
\author[O. Bernardi]{Olivier Bernardi}
\address{Dept. of Mathematics, MIT, 77 Massachusetts Avenue, Cambridge MA 02139, USA}
\email{bernardi@math.mit.edu}
\author[\'E. Fusy]{\'Eric Fusy}
\address{LIX, \'Ecole Polytechnique, 91128 Palaiseau Cedex, France}
\email{fusy@lix.polytechnique.fr}

\thanks{Both authors supported by the European project ExploreMaps -- ERC StG 208471. First author supported by French ANR project \emph{A3}.}

\keywords{Schnyder woods, Schnyder labelling, Forest decomposition, Nash-Williams theorem, 4-regular map,  straight-line drawing, orthogonal drawing}
\date{\today}

\newcommand{\ite}{\noindent $\bullet~$}

\newcommand{\fig}[3]{\begin{figure}[h]\begin{center}\includegraphics[#1]{#2}\end{center}\caption{#3}\label{fig:#2}\end{figure}}

\newcommand{\om}{\omega}
\newcommand{\al}{\alpha}
\newcommand{\be}{\beta}
\newcommand{\ga}{\gamma}

\newcommand{\eps}{\epsilon}

\newcommand{\ee}{\mathrm{e}}
\newcommand{\vv}{\mathrm{v}}
\newcommand{\pp}{\mathrm{p}}
\newcommand{\pf}{\mathrm{part}}
\newcommand{\ff}{\mathrm{full}}

\newcommand{\mR}{\mathcal{R}}

\newcommand{\mE}{\mathcal{E}}
\newcommand{\mF}{\mathcal{F}}
\newcommand{\mT}{\mathcal{T}}

\newcommand{\NN}{\mathbb{N}}

\newcommand{\PP}{\mathbb{P}}

\def\cL{\mathcal{L}}
\def\cO{\mathcal{O}}
\def\cS{\mathcal{S}}

\def\qedclaim{\hfill$\triangle$\smallskip}

\begin{document}

\begin{abstract}
\emph{Schnyder woods} are decompositions of simple triangulations into three edge-disjoint spanning trees  crossing each other in a specific way. In this article, we generalize the definition of  Schnyder woods to $d$-angulations (plane graphs with faces of degree $d$) for all $d\geq 3$. A \emph{Schnyder decomposition} is a set of $d$ spanning forests crossing each other in a specific way, and such that each internal edge is part of exactly $d-2$ of the spanning forests.
We show that a Schnyder decomposition exists if and only if the girth of the $d$-angulation is $d$. As in the case of Schnyder woods ($d=3$), there are alternative  formulations in terms of orientations
(``fractional'' orientations when $d\geq 5$) and in terms of corner-labellings.   Moreover, the set of Schnyder decompositions of a fixed $d$-angulation of girth $d$ has a natural structure of distributive lattice.
We also study the \emph{dual} of Schnyder decompositions which are defined on $d$-regular plane graphs of mincut $d$ with a distinguished vertex $v^*$: these are sets of $d$ spanning trees rooted at $v^*$ crossing each other in a specific way and such that each edge not incident to $v^*$ is used by two trees in opposite directions.
 Additionally, for even values of $d$, we show that a subclass of Schnyder decompositions, which are called even, enjoy additional properties that yield
 a reduced formulation; in the case $d=4$, these correspond to well-studied structures on simple quadrangulations (2-orientations and partitions into 2 spanning trees).

In the case $d=4$, we obtain straight-line and orthogonal planar drawing algorithms by using the dual of even Schnyder decompositions.
 For a 4-regular plane graph $G$ of mincut $4$ with a distinguished vertex $v*$ and $n-1$ other vertices, our algorithms places the   vertices of $G\backslash v^*$  on a $(n-2) \times (n-2)$ grid according to a permutation pattern, and in the orthogonal drawing
 each of the $2n-4$ edges of $G\backslash v^*$ has exactly one bend. 
The vertex $v^*$ can be embedded at the cost of 3 additional rows and columns, and 8 additional bends.
We also describe a further compaction step for the drawing algorithms
 and show that the obtained grid-size is strongly concentrated around
 $25n/32\times 25n/32$ for a uniformly random instance with $n$ vertices.
\end{abstract}

\maketitle

\section{Introduction}
A \emph{plane graph} is a connected planar graph drawn in the plane in such a way that the edges do not cross. A \emph{triangulation} is a plane graph in which every face has degree 3. More generally, a $d$-\emph{angulation} is a plane graph such that every face has degree $d$. In~\cite{Schnyder:wood1,Schnyder:wood2}, Schnyder defined a structure for triangulations which became known as \emph{Schnyder woods}.
Roughly speaking, a Schnyder wood of a triangulation is a partition of the edges of the triangulation into three spanning trees with specific incidence relations;
see Figure~\ref{fig:triangulation} for an example, and Section~\ref{sec:Definitions} for a precise definition.
Schnyder woods have been extensively studied and have numerous applications;
they provide a simple planarity criterion~\cite{Schnyder:wood1}, yield beautiful straight-line drawing
algorithms~\cite{Schnyder:wood2}, have connections with several well-known
lattices~\cite{OB:Catalan-intervals-realizers},
and are a key ingredient in bijections
for counting triangulations~\cite{Poulalhon:triang-3connexe+boundary,FuPoScL}.\\

\fig{width=\linewidth}{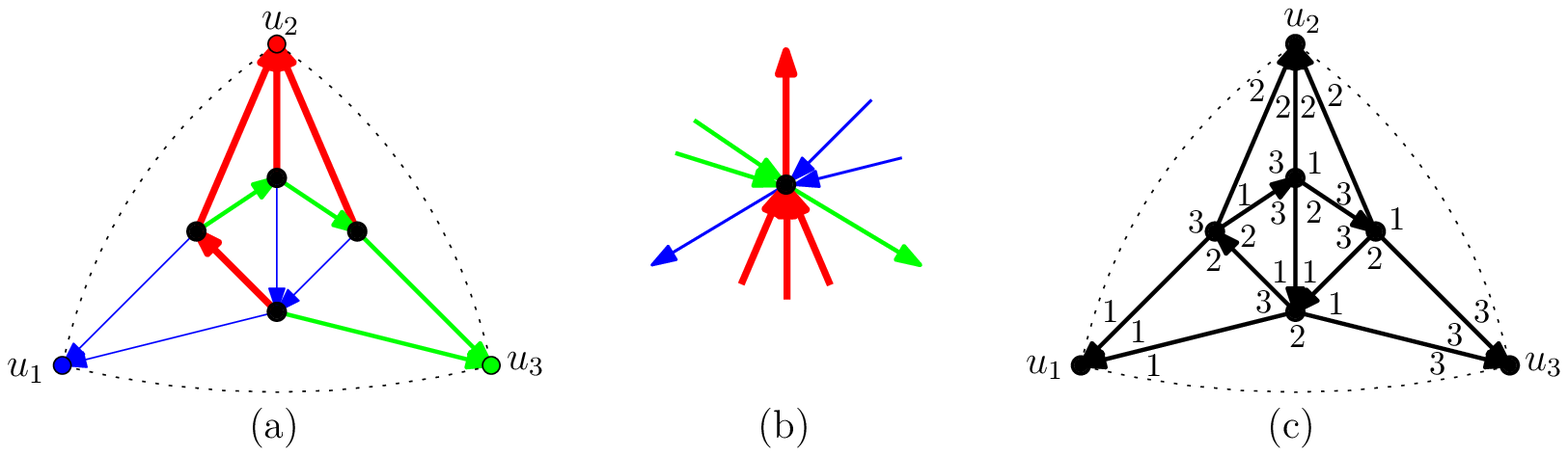}{(a) A Schnyder wood of a triangulation. (b) Crossing rule for the three spanning trees at an internal vertex. (c) Corresponding 3-orientation and clockwise labelling.}

In this article, we generalize the definition of Schnyder woods to
$d$-angulations, for any $d\geq 3$. Roughly speaking, a \emph{Schnyder decomposition} of a $d$-angulation $G$ is a covering of the internal edges of~$G$
by $d$ forests $F_1,\ldots,F_d$,  with specific crossing relations
and such that each internal edge belongs
to exactly $d-2$ of the forests; see Figure~\ref{fig:pentagulation} for an example, and Section~\ref{sec:equiv-definitions} for a precise definition. We show that a $d$-angulation admits a Schnyder decomposition if and only if it has girth $d$ (i.e., no cycle of length smaller than $d$).

\textbf{Other incarnations.}
One of the nice features of Schnyder woods on triangulations
 is that they have two alternative formulations: one in terms of orientations
 with outdegree $3$ at each internal vertex (so-called \emph{3-orientations})
 and one in terms of certain corner-labellings with labels in $\{1,2,3\}$ which
 we call \emph{clockwise labellings},
 see Figure~\ref{fig:triangulation}(c).
We show in Section~\ref{sec:equiv-definitions}
that the same feature occurs for any $d\geq 3$.
Precisely, on a fixed $d$-angulations of girth $d$, we give bijections between Schnyder decompositions, some generalized orientations called $d/(d-2)$-\emph{orientations}, and certain \emph{clockwise labellings} of corners with colors in $\{1,\ldots,d\}$. The $d/(d-2)$-orientations recently appeared in~\cite{Bernardi-Fusy:dangulations}, as a key
ingredient in a bijection to count $d$-angulations of girth $d$.
They are also a suitable formulation to show
that the set of Schnyder decompositions on a fixed $d$-angulation of girth $d$
is a distributive lattice (Proposition~\ref{cor:lattice}).

\fig{width=\linewidth}{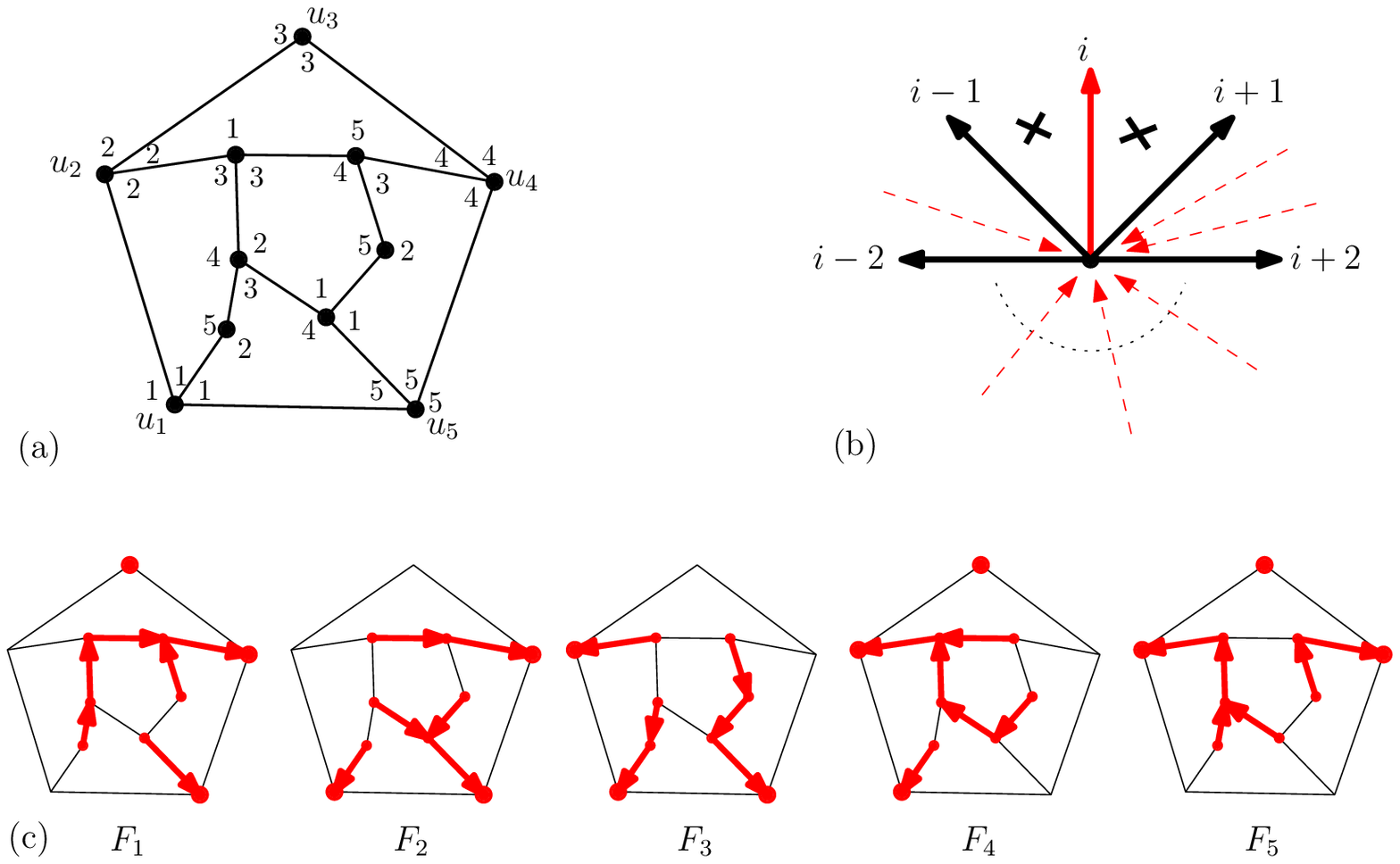}{(a) clockwise labelling of a 5-angulation. (b) Crossing rule for the tree $i$ at an internal vertex. (c) The five forests forming the Schnyder decomposition.}

\textbf{Duality}. Schnyder decompositions can also be studied in a dual setting.
In Section~\ref{sec:duality} we show that if $G$ is a $d$-angulation of girth $d$ and $G^*$ is the dual graph
(which is a $d$-regular plane graph of mincut $d$ with a \emph{root-vertex} $v^*$),
then the combinatorial structures of $G^*$
 dual to the Schnyder decompositions of $G$ are $d$-tuples of spanning trees $T_1^*,\ldots,T_d^*$ such that every edge incident to $v^*$ belongs to one spanning tree, while the other edges belong to two spanning trees in opposite directions;
 in addition around a non-root vertex $v$ the edges leading $v$ to its parent in  $T_1^*,\ldots,T_d^*$ appear in clockwise order\footnote{Actually, the existence of $d$ spanning trees such that every edge non incident to the root-vertex is used twice is granted by the Nash-Williams Theorem (even for non-planar $d$-regular graphs of mincut $d$). Additionally, the decomposition can be taken in such a way that every edge is used once in each direction (for the orientation of the trees toward the root-vertex) by a ``directed version'' of the Nash-Williams theorem due to Edmonds~\cite{Edmonds:decomposition-forests}. However, these theorems do not grant any crossing rule for the spanning trees.}. A duality property was well-known in the case $d=3$, operating
 even more generally on 3-connected plane graphs~\cite{F01}.

\textbf{Bipartite case}. In Section~\ref{sec:even}, we show that when $d$ is even,
$d=2p$, there is a subclass of Schnyder decompositions that
enjoy additional properties. The so-called \emph{even Schnyder decompositions}
 can be reduced to  $p$-tuples of spanning forests with specific incidence relations
 and such that each internal edge belongs to $p-1$ forests; see Figure~\ref{fig:hexangulation} for an example and Theorem~\ref{thm:compact-even-Schnyder} for precise properties. The dual structures are also
characterized, in a reduced form they yield a partition of the edges (except for $p$ of them) of the dual graph
 into $p$ spanning trees oriented toward the root-vertex.
In the case of quadrangulations
  ($p=2$) even Schnyder decompositions correspond to specific
  partitions into two non-crossing  spanning trees; these partitions have been
  introduced in~\cite{Fraysseix:Topological-aspect-orientations}
  under the name of \emph{separating decompositions} and
  are in bijection with
   2-\emph{orientations} (i.e., orientations of internal edges such that every internal vertex  has outdegree $2$).
   The existence of 2-orientations on simple quadrangulations
   is also proved in~\cite{Mendez:these,Fraysseix:Topological-aspect-orientations}.

\textbf{Graph drawing.}
Finally, as an application, we present  in Section~\ref{sec:drawing}
a linear-time orthogonal drawing algorithm for 4-regular plane graphs of mincut $4$, which
 relies on the dual of even Schnyder decompositions.
 For a $4$-regular plane graph $G$ of mincut $4$, with $n$ vertices ($2n$ edges, $n+2$ faces) one of which is called the \emph{root-vertex}, the algorithm places each non-root vertex of $G$ using face-counting operations, that is, one obtains the coordinate of a vertex $v$ by counting the faces in areas delimited by the paths between $v$ and the root-vertex in the spanning trees of the dual Schnyder decompositions. Such face-counting operations were already used in Schnyder's original procedure on triangulations~\cite{Schnyder:wood2},  and in a recent straight-line drawing algorithm of Barri\`ere and Huemer~\cite{Barriere-Huemer:4-Labelings-quadrangulation} which relies on separating decompositions of quadrangulations and is closely related to our algorithm (see
Section~\ref{sec:eq_line}).
The placement of vertices given by our algorithm is such that each line and column of the grid
$(n-2)\times (n-2)$ contains exactly one vertex\footnote{The placement of vertices thus follows a permutation pattern, and it is actually closely related to a recent bijection~\cite{Bonichon:Baxter-permutations} between Baxter permutations and \emph{plane bipolar orientations} (which are known to be in bijection with even Schnyder decompositions of quadrangulations~\cite{Fraysseix:Topological-aspect-orientations,Felsner:Baxter-family}).},
and each edge has exactly one bend (see Figure~\ref{fig:tetravalent}).
Placing the root-vertex and its $4$ incident edges requires 3 more columns, 3 more rows,
and $8$ additional bends (1 bend for two of the edges and 3 bends for the two other ones).
In addition we present a compaction step allowing one to reduce the grid size.
We show that for a uniformly random instance of size $n$, the grid size after compaction
is strongly concentrated around $25n/32\times 25n/32$ (similar techniques for analyzing the average grid-size were used in~\cite{BGHPS04,Fu07b} for  triangulations).

For comparison with existing orthogonal drawing algorithms: our algorithm for $4$-regular graphs of mincut $4$ with $n$ vertices yields a grid size (width $\times$ height) of $(n+1)\times (n+1)$ in the worst-case (which is
optimal~\cite{Ta91}), and concentrated around $25n/32\times 25n/32$ in the (uniformly) random case,
there is \emph{exactly 1 bend}
per edge except for two edges (incident to the root-vertex) with 3 bends, so there is a total
of $2n+4$ bends. In the less restrictive case of  loopless biconnected $4$-regular graphs  Biedl and Kant's algorithm~\cite{Biedl98abetter} has worst case grid size $(n+1)\times(n+1)$, with at most $2n+4$ bends in total and at most $2$ bends per edge. The worst-case values are the same in Tamassia's algorithm~\cite{Ta87}
(analyzed by Biedl in~\cite{Bi06}) but with the advantage that the total number of bends obtained is best possible for each fixed loopless biconnected $4$-regular plane graphs.  In the more restrictive case of  $3$-connected $4$-regular
 graphs,  Biedl's algorithm~\cite{Bi06}, which revisits an algorithm by Kant~\cite{Ka96}, has worst case grid size $(\tfrac{2}{3}n+1)\times (\tfrac{2}{3}n+1)$, with $\lceil \tfrac{4}{3}n\rceil +4$ bends in total, and at most $2$ bends per edge except for the root-edge having at most 3 bends.
We refer the reader to~\cite{Bi98} for a survey on the worst-case grid size of orthogonal drawings.

\medskip

\section{Plane graphs and fractional orientations of $d$-angulations}\label{sec:Definitions}
In this section we recall some classical definitions about plane graphs, and prove the existence of $d/(d-2)$-orientations for $d$-angulations.

\subsection{Graphs, plane graphs and $d$-angulations}~\\
Here \emph{graphs} are finite; they have no loop but can have multiple edges. For an edge $e$, we use the notation $e=\{u,v\}$ to indicate that $u,v$ are the endpoints
of~$e$. The edge $e=\{u,v\}$ has two \emph{directions} or \emph{arcs}; the arc with origin $u$ is denoted $(u,e)$.
A graph is $d$-\emph{regular} if every vertex has degree $d$.
The \emph{girth} of a graph is the length of a shortest cycle contained in the graph.
The \emph{mincut} of a connected graph (also called \emph{edge-connectivity})
is the minimal number of edges that one needs to delete in order to disconnect the graph.

A \emph{plane graph} is a connected planar graph drawn in the plane in such a way that the edges do not cross.
A \emph{face} is a connected component of the plane cut by the drawing of the graph. 
The numbers $v(G),e(G),f(G)$ of vertices, edges and faces of a plane graph satisfy the \emph{Euler relation:} $v(G)-e(G)+f(G)=2$.
A \emph{corner} is a triple $(v,e,e')$, where $e,e'$ are consecutive edges in clockwise order around $v$.
A corner is said to be \emph{incident} to the face which contains it. 
The \emph{degree} of a face or a vertex is the number of incident corners.
The bounded faces are called \emph{internal} and the unbounded one is called \emph{external}. 
A vertex, an edge or a corner is said to be \emph{external} if it is incident to the external face and \emph{internal} otherwise.

A  $d$-\emph{angulation} is a plane graph where every face has degree $d$. \emph{Triangulations} and \emph{quadrangulations} correspond to $d=3$ and $d=4$ respectively. Clearly, a $d$-angulation has girth at most $d$, and if it has girth $d$ then it has $d$ distinct external vertices (indeed
if two of the $d$ corners in the external face were incident to the same vertex then
one could extract from the contour of $f$ a cycle of length smaller than $d$).
The incidence relation between faces and edges in a $d$-angulation $G$ gives $d\, f(G)=2\,e(G)$. Combining this equation with the Euler relation gives
\begin{equation}\label{eq:incidence}
\frac{e(G)-d}{v(G)-d}=\frac{d}{d-2}.
\end{equation}
 Note that when the external face has $d$ distinct vertices (which is the case when the girth is $d$), the left-hand-side of~\eqref{eq:incidence} is the ratio between internal edges and internal vertices.

\subsection{Existence of $d/(d-2)$-orientations}~\\
Let $G$ be a graph and let $k$ be a positive integer. A $k$-\emph{fractional orientation} of $G$ is a function $\Omega$ from the arcs of $G$ to the set $\{0,1,2,\ldots,k\}$ such that the values of the two arcs constituting an edge $e=\{u,v\}$ add up to $k$: $\Omega(u,e)+\Omega(v,e)=k$. We identify 1-orientations with the classical notion of orientations: each edge $e=\{u,v\}$, with $(u,e)$ the arc of value $\Omega=1$, is directed from $u$ to $v$.
(More generally, it is possible to think of~$k$-fractional orientations as orientations of the graph $k\,G$ obtained from $G$ by replacing each edges by $k$ \emph{indistinguishable} parallel edges.) The \emph{outdegree} of a vertex $v$ is the sum of the value of~$\Omega$ over the arcs having origin $v$.
\begin{Def}
A $j/k$\emph{-orientation} of a plane graph is a $k$-fractional orientation of its internal edges (the external edges are not oriented) such that every internal vertex has outdegree $j$ and every external vertex has outdegree $0$. A $j$\emph{-orientation} is a $j/1$-orientation (which can be seen as a non-fractional orientation of the internal edges).
\end{Def}

Note that, for $n\geq 1$, the $j/k$-orientations identify with a subset of the $(nj)/(nk)$-orientations (those where the value of every arc is a multiple of $n$). Given relation~\eqref{eq:incidence}, it is natural to look for $d/(d-2)$-orientations of~$d$-angulations. For $d=3$, $d/(d-2)$-orientations correspond to the notion of~$3$\emph{-orientations} of triangulations considered in the literature. However, for $d=4$ the $d/(d-2)$-orientations are more general than the 2\emph{-orientations} of quadrangulations considered for instance in~\cite{Bonichon:Baxter-permutations,Felsner:Baxter-family}. The existence of~$3$-orientations for simple  triangulations  and the existence of 2-orientations for simple quadrangulations were proved in~\cite{Schnyder:wood2} and~\cite{Mendez:these} respectively. We now prove a more general existence result.
\begin{thm}\label{thm:exists}
A $d$-angulation admits a $d/(d-2)$-orientation if and only if it has girth $d$. Moreover, if $d=2p$ is even, then any $d$-angulation of girth $d$ also admits a $p/(p-1)$-orientation.
\end{thm}
In order to prove Theorem~\ref{thm:exists}, we first give a general existence criteria for $k$-fractional orientations with prescribed outdegrees. Let $H=(V,E)$ be a graph, let $\alpha: V\to \mathbb{N}$, and let $k\geq 1$.
An $(\alpha,k)$-orientation is a $k$-fractional orientation such that every vertex $v$ has outdegree $\alpha(v)$.
A subset $S$ of vertices is said to be \emph{connected} if the subgraph $G_S=(S,E_S)$ is connected, where $E_S$ is the set of edges with both ends in $S$.

\begin{lem}[Folklore, see~\cite{Felsner:lattice}]\label{lem:condition-existence}
Let $H=(V,E)$ be a graph and let $\al:V\mapsto \NN$. There exists an orientation of $H$ such that every vertex $v$ has outdegree $\al(v)$ (i.e., an $(\alpha,1)$-orientation) if and only if
\begin{enumerate}
\item[(a)] $\sum_{v\in V}\al(v)=|E|$
\item[(b)] for all connected subset of vertices $S\subseteq V$, $\sum_{v\in S}\al(v)\geq |E_S|$, where $E_S$ is the set of edges with both ends in $S$.
\end{enumerate}
\end{lem}


\begin{cor}\label{cor:condition-existence}
Let $H=(V,E)$ be a graph and let $\al:V\mapsto \NN$. There exists a  $(\al,k)$-orientation of $H=(V,E)$ if and only if
\begin{enumerate}
\item[(a)] $\sum_{v\in V}\al(v)=k\,|E|$
\item[(b)] for all connected subset of vertices $S\subseteq V$, $\sum_{v\in S}\al(v)\geq k\,|E_S|$, where $E_S$ is the set of edges with both ends in $S$.
\end{enumerate}
\end{cor}

\begin{proof}[Proof of Corollary~\ref{cor:condition-existence}]
The conditions are clearly necessary. 
We now suppose that the mapping $\al$ satisfies the conditions and want to prove that an $(\al,k)$-orientation exists. 
For this, it suffices to prove the existence of an orientation of the graph $k\,H$ (obtained from $H$ by replacing each edge by $k$ edges in parallel) with outdegree $k\al(v)$ for each vertex $v$. This is granted by Lemma~\ref{lem:condition-existence}.
\end{proof}

\begin{proof}[Proof of Theorem~\ref{thm:exists}]
We first show that having girth $d$ is necessary. Let $G$ be a $d$-angulation admitting a $d/(d-2)$-orientation. Suppose  for contradiction that $G$ has a simple cycle $C$ of length $\ell< d$. Let $G'$ be the plane subgraph made of the cycle $C$ and all the edges and vertices lying inside $C$. The Euler relation and the incidence relation between faces and edges of the plane graph $G'$ imply that the numbers $\vv'$ and $\ee'$ of vertices and edges lying strictly inside $C$ (not on $C$) satisfy $d\vv'=(d-2)\ee'+d-\ell>(d-2)\ee'$. This yields a contradiction since $d\vv'$ is the sum of the required outdegrees for the vertices lying strictly inside $C$, while $(d-2)\ee'$ is the sum of the values of the arcs incident to these vertices.

We now prove that having girth $d$ is sufficient. Let $G=(V,E)$ be a $d$-angulation of girth $d$. Since $G$ has girth $d$, it has $d$ distinct external vertices and $d$ distinct external edges. We call \emph{quasi $d/(d-2)$-orientation} of $G$ a $(d-2)$-fractional orientation of all the edges of $G$ (including the external ones) such that every internal vertex has outdegree $d$ and every external vertex has outdegree $d-2$. We first prove the existence of a quasi $d/(d-2)$-orientation for $G$ using  Corollary~\ref{cor:condition-existence}. First of all, Equation \eqref{eq:incidence} shows that Condition (a) holds. We now consider a connected subset of vertices $S$, and the induced plane subgraph $G_S=(S,E_S)$. Since $G$ has girth $d$, every face of $G_S$ has degree at least $d$. Hence the  Euler relation and the incidence relation  between faces and edges imply that $d\,|S|-2d \geq (d-2)\, |E_S|$. Let $r$ be the number of external vertices of $G$ in $S$. Since $r\leq d$, we get $d\,|S|-2r\geq (d-2)\, |E_S|$, which is exactly Condition (b) for the connected set $S$. Hence $G$ admit a quasi $d/(d-2)$-orientation $O$. We now consider the orientation $O'$ obtained from $O$ by forgetting the orientation of the external edges, so that $O'$ is a $k$-fractional orientation of the internal edges such that the outdegree of every internal vertex is $d$. Let $s$ and $s'$ be the sum of the outdegrees of the external vertices in $O$ and $O'$ respectively. By definition of  quasi $d/(d-2)$-orientations, $s=(d-2)\, d$ (because each external vertex has outdegree $d-2$), and  $s-s'=(d-2)\, d$ (because each external edge contributes  $d-2$ to $s$). Thus $s'=0$, which ensures that  $O'$ is a  $d/(d-2)$-orientation.
 
The ingredients are exactly the same to show that a $2p$-angulation admits a $p/(p-1)$-orientation if and only if it has girth $2p$.
\end{proof}

\section{Three incarnations of Schnyder decompositions}\label{sec:equiv-definitions}
In this section we define \emph{Schnyder decompositions} and  certain labellings of the corners which we call \emph{clockwise labellings}.
We show that on a fixed $d$-angulation $G$ of girth $d$, Schnyder decompositions, clockwise labellings and $d/(d-2)$-orientations
are in bijective correspondence;
hence the three definitions are actually three incarnations of the same structure.

 We start with the definition of Schnyder decompositions, which is illustrated in Figure~\ref{fig:pentagulation}.
In the following, $d$ is an integer greater or equal to 3 and $G$ is a $d$-angulation with $d$ distinct external vertices $u_1,\ldots,u_d$ in clockwise order around the external face. We will denote by $[d]$ the set of integers $\{1,\ldots,d\}$ \emph{considered modulo $d$} (i.e., the addition and subtraction operations correspond
to cyclic shifts).
\begin{Def}\label{def:Schnyder}
A \emph{Schnyder decomposition} of~$G$ is a covering of the internal edges of $G$ by $d$ oriented forests $F_1,\ldots,F_d$
(one forest for each color $i\in [d]$) such that
\begin{enumerate}
\item[(i)] Each internal edge $e$ appears in $d-2$ of the forests.
 \item[(ii)] For each $i$ in $[d]$, the forest $F_i$ spans all vertices except $u_i,u_{i+1}$;
it is made of $d-2$ trees each containing one of the external vertices $u_j,j\neq i,i+1$, and the tree containing $u_j$ is oriented toward $u_j$ which is considered as its \emph{root}.
\item[(iii)] Around any internal vertex $v$, the outgoing arcs of colors $1,\ldots,d$ appear in clockwise order around $v$ (some of these arcs can be on the same edge). Moreover, for $i\in[d]$,
     calling $e_i$ the edge bearing the outgoing arc of color $i$, the incoming arcs of color $i$ are strictly between $e_{i+1}$ and $e_{i-1}$ in clockwise order around $v$.
\end{enumerate}
We denote the decomposition by $(F_1,\ldots,F_d)$.
\end{Def}


The classical definition of \emph{Schnyder woods} coincides with our definition of \emph{Schnyder decomposition} for triangulations (case $d=3$).
The \emph{colors} of an edge are the colors of the $d-2$ forests containing this edge, the \emph{colors}
 of an arc $(u,e)$ (also called \emph{colors going out of $u$ along $e$})
 denote the colors of the oriented forests
 using this arc.
The \emph{missing colors of an edge}
are the colors of the two forests not containing this edge.

\begin{Remark} \label{rk:immediate-csq-ii}
Condition (iii) in Definition~\ref{def:Schnyder} immediately implies two properties:
\begin{enumerate}
\item[(a)] If an internal edge has $i,j$ as missing colors, then the colors $i+1, i+2,\ldots ,j-1$ are all in one direction of~$e$, while the colors $j+1, j+2,\ldots ,i-1$ are all in the other direction of~$e$.
\item[(b)] Let $v$ be an internal vertex and $i\in[d]$ be a color. Let $e,e',f,f'$ be edges incident to $v$ with $e,e'$ of color $i$ and $f,f'$ of color $i+1$. If $e,f$ are incoming, $e',f'$ are outgoing and $e,e',f$ appear in clockwise order around $v$ (with possibly $e=f$), then the outgoing edge $f'$ appears between $e'$ and $f$ in clockwise order around $v$ (with possibly $e'=f'$). As a consequence, \emph{a directed path of color $i+1$ cannot cross a directed path of color $i$ from right to left} (however, a crossing from left to right is possible).
\end{enumerate}
\end{Remark}

We now define \emph{clockwise labellings}. An example is shown in Figure~\ref{fig:pentagulation}(a).

\begin{Def}
A \emph{clockwise labelling} of~$G$ is the assignment  to each corner of a color, or \emph{label} (we use the terms ``color'', and ``label'' synonymously), in $[d]$ such that:
\begin{enumerate}
\item[(i)] The colors $1,2,\ldots,d$ appear in clockwise order\footnote{The root-face, which is unbounded, has a special role; when we say
    ``in clockwise order around the root-face''
    we mean that we walk along
     the contour of the root-face with the root-face on the left. In contrast, when we
    say ``in clockwise order'' around a non-root face $f$ we mean that we walk along the contour of
    $f$ with $f$ on our right.} around each face of~$G$.
\item[(ii)] For all $i$ in $[d]$, the corners incident to the external vertex $u_i$ are of color $i$.
\item[(iii)] In clockwise order around each internal vertex, there is exactly one corner having a larger color than the next corner.
\end{enumerate}
\end{Def}

\begin{thm}\label{thm:equiv} Let $G$ be a $d$-angulation. The sets of Schnyder decompositions, clockwise labellings, and
$d/(d-2)$-orientations of~$G$ are in bijection.
\end{thm}
Given the existence result in Theorem~\ref{thm:exists}, we obtain:
\begin{cor}
A $d$-angulation $G$ admits a Schnyder decomposition (respectively, a clockwise labelling) if and only if it has girth $d$.
\end{cor}

In the next two subsections we prove Theorem \ref{thm:equiv} and show bijections between Schnyder decompositions, clockwise labellings, and  $d/(d-2)$-orientations. Then, we present a lattice structure on the set of Schnyder decompositions of a given $d$-angulations.

\subsection{Bijection between clockwise labellings and $d/(d-2)$-orientations}\label{sec:ddmt}~\\
Let $L$ be a clockwise labelling of $G$. For each arc $a=(u,e)$
of $G$, the \emph{clockwise-jump across $a$} is the quantity in $\{0,\ldots,d-1\}$ equal to $\ell_2-\ell_1$ modulo $d$,
where $\ell_1$ and $\ell_2$ are respectively
the colors of the corners preceding and following $e$ in clockwise order around $u$.

\begin{lem}\label{lem:bijLO}
For each vertex $v$ of $L$, the sum of the clockwise-jumps over
the  arcs incident to $v$ is $d$ if $v$ is an internal vertex and is $0$
if $v$ is an external vertex.
For each internal edge $e=\{u,v\}$ of $L$, the sum of the clockwise-jumps on
the two arcs constituting $e$ is $d-2$.
\end{lem}
\begin{proof}
The first assertion is a direct consequence of Properties~(ii),~(iii) of clockwise labellings. This assertion implies that the sum of clockwise-jumps over all arcs of $G$ is $d\cdot n$, where $n$ is the number of internal vertices of $G$.
By \eqref{eq:incidence} the number of internal edges is $m=\tfrac{d}{d-2}n$, so the sum of clockwise-jumps over all
arcs of $G$ equals $(d-2)m$.
Let us now consider the sum $S(e):=J(v,e)+ J(u,e)$ for an internal edge $e=\{u,v\}$. Let $i,j$ denote the colors of the corners preceding and following the edge $e$ in clockwise order around $u$. By Property~(i) of clockwise labellings, the colors preceding and following $e$ in clockwise order around $v$ are $j+1$ and $i-1$ respectively. Hence if $J(u,e)=d-1$, then $S(e)=2d-2$, and otherwise $S(e)=d-2$.
Given that the sum of the $J(v,e)$ over all the arcs on internal edges of $G$ is equal to $(d-2)m$,  we conclude  that
 $S(e)=d-2$ for all internal edges.
\end{proof}

Lemma~\ref{lem:bijLO} implies that in a clockwise labelling a clockwise-jump
can not exceed $d-2$, which ensures that our definition of clockwise labellings
coincides with the Schnyder labellings~\cite{Schnyder:wood1} in the case $d=3$.
It also ensures that the orientation $O$ with value $\Omega(a):=J(a)$ for each arc $a$ is a $d/(d-2)$-orientation. 
The mapping associating $O$ to $L$
is called $\Psi$. Denote by $\cL$ the set of clockwise labellings of $G$
and by $\cO$ the set of $d/(d-2)$-orientations of $G$.

\begin{prop}\label{prop:psi}
The mapping $\Psi$ is a bijection between $\cL$ and $\cO$.
\end{prop}
\begin{proof}
 First of all, it is clear that $\Psi$ is injective, since the values of $\Omega$ suffice to recover  the color of every corner by starting from the corners incident to external vertices (whose color is known by definition) and propagating the colors according to the rule
 (considering that the arc values on external edges are $0$):
\begin{itemize}
\item[(1)] if the color of a corner is $i$, the color of the next corner in clockwise order around its face is $i+1$,
\item[(2)] if the color of a corner $c$ is $i$, with $v$ the vertex incident to $c$, then the color of the next corner $c'$ in clockwise order around $v$ is $i+\Omega(v,e)$, where $e$ is the edge between $c$ and $c'$.
\end{itemize}
If $\Omega = \Psi(L)$ these rules uniquely determine $L$.
We will now prove that it is  possible to apply these rules starting from \emph{any}
$d/(d-2)$-orientation $\Omega$ without encountering any conflict (thereby proving the surjectivity of $\Psi$). Let $\Omega$ be a $d/(d-2)$-orientation.
Let $H$ be the plane graph, called \emph{corner graph}, whose vertices are the internal corners of $G$ and whose edges are the pairs $\{c,c'\}$ of corners which are consecutive around a vertex or a face of $G$,
see Figure~\ref{fig:corner-graph}. We define a function $\om$ on the arcs of the corner graph $H$ as follows. For an arc $a$ of $H$ going from a corner $c$ to a corner $c'$, we set $\om(a)=1$ (resp. $\om(a)=-1$)  if the corner $c'$ follows $c$ in clockwise (resp. counterclockwise) order around a face of $G$, and we set  $\om(a)=\Omega(v,e)$ (resp. $\om(a)=-\Omega(v,e)$) if $c'$ follows $c$ in clockwise (resp. counterclockwise) order around a vertex $v$ of $G$, where $e$ is the edge of $G$ between $c$ and $c'$.
For each directed path $P=c_1,\ldots,c_r$ in $H$, the \emph{arcs composing $P$}
are the arcs from $c_i$ to $c_{i+1}$ for $1\leq i<r$. We now define the function $\om$ on the directed paths of $H$ by setting $\om(P)=\sum_a\om(a)$, where the sum is over the arcs composing the path $P$.
Consider a directed path $P$ of $H$ starting at a corner $c$ and ending at a corner $c'$. By definition, if we  use the rules (1) (2) in order to attribute a colors to every corner of $H$, the colors $i$ and $i'$ of the corner $c$ and $c'$ will satisfy $i'=i+\om(P)$ modulo $d$. Hence,  to show that there is no conflict when propagating the colors according to the rules (1) (2) we have to show that any pair of directed paths $P,P'$ of $H$ with the same endpoints satisfy $\om(P)=\om(P')$ modulo $d$. Equivalently, we have to show that for each directed cycle $C$ of $H$, we have $\om(C)=0$ modulo $d$, and the verification can actually be restricted to simple directed cycles. Let $C$ be a simple directed cycle in $H$. Observe that the graph $H$ can be naturally superimposed
with $G$ (see Figure~\ref{fig:corner-graph}), revealing that each internal face of $H$
corresponds either to an internal vertex, edge, or face of $G$.
If $C$ is the cycle delimiting a face $f$ of $H$ in clockwise direction,
the value $\om(C)$ is $d$ (resp. $-d$, $d$)
if $f$ corresponds to an internal vertex (resp. edge, face) of $G$. More generally,
if $C$ is a simple clockwise (resp. counterclockwise) directed cycle of $H$, $\om(C)=\sum_{f}\om(C_f)$ (resp. $\om(C)=-\sum_{f}\om(C_f)$) where the sum is over the faces of $H$ enclosed in $C$, since the contributions of arcs strictly inside $C$ cancel out in this sum.  Thus, the value of $\om(C)$ is a multiple of $d$ for any directed cycle $C$. 
Thus, there is no possible conflict  when propagating the colors according to the rules (1) and (2): one can define the color of every internal corner of $G$ by setting the color of one of the corner incident to the external vertex $u_1$ to be 1, and asking for the rules (1) (2) to hold everywhere. Observe lastly these rules will attribute the color $i$ to every internal corner incident to the external vertex $u_i$ for all $i=1\ldots d$. Hence the coloring $L$ obtained is a clockwise labellings of $G$, and $\Psi(L)=\Omega$. 
This completes the proof that the mapping $\Psi$ is surjective, hence bijective.
\end{proof}


\begin{figure}[h]
\begin{minipage}{1.1\linewidth}
\hspace{-.8cm}
\begin{minipage}{.4\linewidth}
\centerline{\includegraphics[scale=.65]{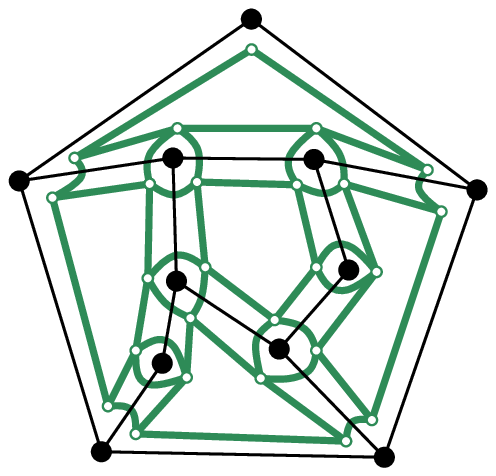}}
\caption{The corner graph.}\label{fig:corner-graph}
\end{minipage}\hspace{-.5cm}
\begin{minipage}{.65\linewidth}
\centerline{\includegraphics[scale=.6]{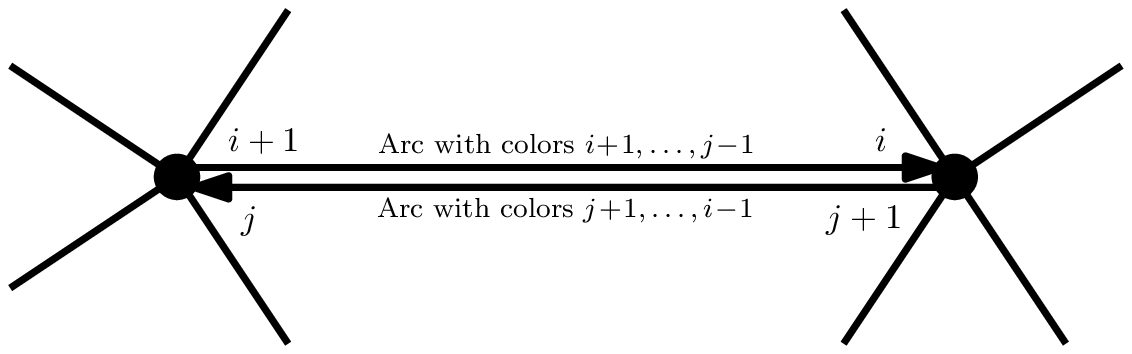}}\vspace{.4cm}
\caption{Mapping $\Phi$ from clockwise labellings to Schnyder decompositions.}\label{fig:rule-Phi}
\end{minipage}
\end{minipage}
\end{figure}


\subsection{Bijection between clockwise labellings and Schnyder decompositions}
We now define a mapping from clockwise labellings to Schnyder decompositions. Let $L$ be a clockwise labelling of $G$.
For each internal arc $a=(u,e)$ of $G$, we give to $a$ the colors $i,i+1,\ldots j-1$ (no color if $i=j$), where $i$ and $j$ are respectively the colors of the corners preceding and following $e$ in clockwise order around $u$, see Figure \ref{fig:rule-Phi}. For $i\in[d]$ we denote by $F_i$ the oriented graph of color $i$, and we call $\Phi$ the mapping that associates $(F_1,\ldots,F_d)$ to $L$.

\begin{lem}
For each $L\in\cL$, $(F_1,\ldots,F_d)=\Phi(L)$ is a Schnyder decomposition.
\end{lem}
\begin{proof}
According to Lemma~\ref{lem:bijLO}, each internal edge receives a total of $d-2$ colors (see Figure~\ref{fig:rule-Phi}). Property (iii) of clockwise labellings ensures that every internal vertex $v$ has exactly
one outgoing edge in each of the oriented subgraphs $F_1,F_2,\ldots,F_d$, and that these edges appear in clockwise order around $v$. Moreover, if an edge $e$ has  color $k\in[d]$ on the arc
directed toward an internal vertex $v$, then in clockwise order around $v$ the edge $e$ is preceded and followed by corners with colors $i+1,j$ satisfying $k\in\{j+1,\ldots,i-1\}$. Hence, $i+1\notin\{k,k+1\}$ and $e$ is not between the outgoing edges of color $k-1$ and $k+1$ in clockwise order around $v$.
This proves Property~(iii) of Schnyder decompositions.

It remains to show that the oriented subgraphs $F_1,\ldots,F_d$ are forests oriented toward the external vertices. Suppose that for some $i\in[d]$, $F_i$ is not a forest oriented toward the external vertices. Since every internal vertex has exactly one outgoing edge in $F_i$, this implies the existence of a directed cycle in $F_i$. Consider a directed cycle $C$ in one of the subgraphs $F_1,F_2,\ldots,F_d$ enclosing a minimal number of faces. Clearly, the cycle $C$ is simple and has only internal vertices. Suppose first that the cycle $C$ is clockwise. There exists a color $i\in[d]$ such that all the arcs of $C$ have the color $i$, but not all the arcs have the color  $i+1$. There exists a vertex $v$ on $C$ such that the edge $e_{i+1}$ of color $i+1$ going out of $v$ is not equal to the the edge $e_i\in C$ of color $i$ going out of $v$. In this case, the edge $e_{i+1}$ is strictly inside $C$. Indeed, we have established above that the incoming arcs of color $i$ at $v$ (one of which belongs to $C$) are strictly between the outgoing edges of color $i-1$ and $i+1$ in clockwise direction around $v$. Let $v_1$ be the other end of the edge $e_{i+1}$. Since $v_1$ is an internal vertex, there is an edge of color $i+1$ going out of $v_1$ and leading to a vertex $v_2$. Continuing in this way we get a directed path $P$ of color $i+1$ starting at $v$ and either ending at an external vertex or creating a directed cycle of color $i+1$. Moreover, Property~(b) in Remark~\ref{rk:immediate-csq-ii}
 implies that $P$ remains inside the cycle $C$ of color $i$. This means that the path $P$ creates a directed cycle of color $i+1$ which is strictly contained inside $C$. This cycle encloses less faces than $C$, contradicting our minimality assumption on $C$. Similarly, if the cycle $C$ is directed counterclockwise there is a cycle (of color $i-1$) which encloses less faces than $C$. We thereby reach a contradiction. This proves that there is no monochromatic directed cycle. Therefore  $F_1,\ldots,F_d$ are forests directed toward the external vertices.

In order to complete our proof that $(F_1,\ldots,F_d)$ is a Schnyder decomposition, it only remains to show that for all $i\in [d]$ the forest $F_i$ is not incident to the external vertices  $u_i,u_{i+1}$. For this, recall that the corners incident to the external vertex $u_i$ have color $i$, so that for any internal edge  $e=\{u_i,v\}$ the two corners incident to $v$ have colors $i-1$ and $i+1$. Hence, $e$ belongs to $F_{i+1},F_{i+2},\ldots,F_{i-2}$ but not to $F_{i-1}$ or $F_{i}$. Thus, $(F_1,\ldots,F_d)$ is a Schnyder decomposition.
\end{proof}

Let $\cS$ be the set of Schnyder decompositions of $G$.
\begin{prop}
The mapping $\Phi$ is a bijection between $\cL$ and $\cS$.
\end{prop}
\begin{proof}
It is clear from the definition of the mappings $\Psi$ and $\Phi$ that the orientation $O=\Psi(L)$ associated with a clockwise labelling $L$ is obtained from the Schnyder decomposition $(F_1,\ldots,F_d)=\Phi(L)$ by forgetting the colors. Denoting by $\Gamma$ the ``color-deletion'' mapping, we thus have $\Psi=\Gamma\circ\Phi$. By Proposition \ref{prop:psi}, the mapping $\Psi$ is injective, thus $\Phi$ is also injective.

To prove that $\Phi$ is surjective, we consider a Schnyder decomposition $S=(F_1,\ldots,F_d)$. Since $\Gamma(S)$  is a $d/(d-2)$-orientation, $L=\Psi^{-1}\circ \Gamma(S)$ is a clockwise labelling. We now show that  $\Phi(L)=S$ (thereby showing the surjectivity of $\Phi$). Let $S'=\Phi(L)$.  First observe that for any arc $a$ of $G$ the number of colors of $a$ in $S$ and in $S'$ is the same: it is equal to the clockwise jump across $a$ in $L$. We say that $S$ and $S'$ \emph{agree} on an arc $a=(u,e)$ if the colors of $a$ are the same in $S$ and in $S'$.  We say that $S$ and $S'$ \emph{agree} on an edge $e$ if they agree on both arcs of $e$.  For all $i\in[d]$, Property (ii) of Schnyder decompositions implies that the internal edges incident to the external vertex $u_i$ have missing colors $i$ and $i-1$ in both $S$ and $S'$ (and all their colors are oriented toward $u_i$), hence  $S$ and $S'$ agree on these edges.  We now suppose that $v$ is an internal vertex incident to an edge $e$ on which $S$ and $S'$ agree and show that in this case $S$ and $S'$ agree on every edge incident to $v$ (this shows that $S$ and $S'$ agree on every edge). Suppose first that the edge $e$ has some colors going out of $v$ (which by hypothesis are the same in $S$ and in $S'$). In this case, $S$ and $S'$ agree on each arc going out of $v$ (since the colors going out of $v$ are $1,2,\ldots,d$ in clockwise order).
Now, for an edge $e'=\{v,v'\}$, either the arc $a=(v,e')$ has some colors $i+1,\ldots, j-1$ in which case the arc $a'=(v',e')$ has the colors  $j+1,\ldots, i-1$ (both in $S$ and in $S'$), or the arc  $a=(v,e')$ has no color in which case  the $d-2$ colors of $a'=(v',e')$ are imposed (both in $S$ and in $S'$) by Property (iii) of clockwise labellings (which implies that the missing colors are $i$ and $i+1$ for each edge $e'$ strictly between the outgoing arcs of color $i$ and $i+1$ around $v$). We now suppose that the edge $e$ has no color going out of $v$. Let $i,i+1$ be the missing colors of $e$. Again by Property~(iii) of Schnyder decompositions,
the edge $e$ is between the outgoing arcs of color $i$ and $i+1$ around $v$ (both for $S$ and $S'$), thus
$S$ and $S'$ agree on these arcs. Hence $S$ and $S'$  also agree on all the edges incident to $v$ by the same reasoning as above. This shows that $S=S'$ hence that $\Psi$ is surjective.
\end{proof}

\subsection{Lattice property}~\\
In this subsection we explain how to endow the set of $d/(d-2)$-orientations  of a $d$-angulation (equivalently its set of Schnyder decompositions) with the structure of a distributive lattice, a property already known for $d=3$~\cite{Brehm00,Mendez:these}.

Let $G$ be a plane graph. Given a simple cycle $C$, the \emph{clockwise} (resp. \emph{counterclockwise})
 arcs of $C$ are the arcs  in clockwise (resp. counterclockwise) direction around $C$. Given a $k$-fractional orientation $O$, the cycle $C$ is a \emph{counterclockwise circuit} if the value of every counterclockwise arc is positive. We then denote by $O_C$ the orientation obtained from $O$ by increasing  (by $1$) the value of the clockwise arcs of $C$  and decreasing (by $1$) the value of the counterclockwise arcs of $C$. The transition from $O$ to $O_C$ is called the \emph{pushing} of the cycle $C$. We then prove the following lemma as an easy consequence of~\cite{FeKn10}.

\begin{lem}\label{lem:structure}
Let $G=(V,E)$ be a plane graph and let $\alpha$ be a function from $V$ to $\mathbb{N}$.
If the set $A$ of $(\alpha,k)$-orientations of $G$ is not empty, then
the transitive closure of cycle-pushings gives a partial order on $A$ which is a distributive lattice.
\end{lem}
\begin{proof}
First we recall the result from~\cite{FeKn10} (formulated there in the dual setting).
Let $H$ be a loopless directed plane graph with set of vertices $V_H$ and set of (directed) edges $E_H$.
Given a function $\Delta:V_H\to\mathbb{Z}$, a \emph{$\Delta$-bond} is a function $\Omega:E_H\to\mathbb{Z}$
such that for any vertex $v$, $\sum_{e\in\mathrm{out}(v)}\Omega(e)-\sum_{e\in\mathrm{in}(v)}\Omega(e)=\Delta(v)$, where $\mathrm{out}(v)$ and $\mathrm{in}(v)$ are the sets of edges with origin $v$ and end $v$ respectively.
Given two functions $\ell:E_H\to\mathbb{Z}$ and $u:E_H\to\mathbb{Z}$, a $(\Delta,\ell,u)$\emph{-bond} is a $\Delta$-bond $\Omega$ such that $\ell(e)\leq\Omega(e)\leq u(e)$ for each edge $e$. In this context, a simple cycle $C$ is called \emph{admissible} if  each clockwise edge $e$ of $C$ satisfies $\Omega(e)<u(e)$ and
each counterclockwise edge $e$ of $C$ satisfies $\Omega(e)>\ell(e)$.
 \emph{Incrementing} the admissible cycle
$C$ means increasing by $1$ the value of the clockwise edges and decreasing by $1$ the value 
of the counterclockwise edges of $C$. It is shown in~\cite{FeKn10}
that, if the set  of $(\Delta,\ell,u)$-bonds is not empty, the transitive closure of cycle-incrementations gives a distributive lattice.

We now use this result in our context. Let $H$ be the directed plane graph
obtained from the plane graph $G$ by inserting in the middle of each edge $e=\{u,v\}$ of $G$ 
a vertex $v_e$, called an \emph{edge-vertex}, and orienting the two resulting edges
$\{u,v_e\}$ and $\{v,v_e\}$ toward $v_e$. Observe that the arcs of $G$ can be identified with the edges of $H$. Let $\Delta$ be the function defined on the vertices of $H$ by setting $\Delta(v)=\alpha(v)$ for each vertex $v$
of $G$ and $\Delta(v)=-k$ for each edge-vertex.
Clearly, the $(\alpha,k)$-orientations of $G$ correspond bijectively
to the $(\Delta,\ell,u)$-bonds of $H$, where  $\ell(e)=0$ and $u(e)=k$ for each edge $e$ of $H$. Moreover, the cycle-incrementations in $\Delta$-bonds correspond to the cycle-pushings in $(\alpha,k)$-orientations. Hence, the above mentioned result guarantees that the transitive closure of cycle-pushings defines a distributive lattice on the set of $(\alpha,k)$-orientations of $G$. 
\end{proof}

\begin{prop}\label{cor:lattice} 
Let $G$ be a $d$-angulation of girth $d$. 
The set of $d/(d-2)$-orientations 
of $G$ (equivalently, its set of Schnyder decompositions, or clockwise-labellings) can be given the structure of a distributive lattice in which the order relation is the transitive closure
of the cycle-pushings on $d/(d-2)$-orientations.
Moreover, the covering relation corresponds to the pushing of any counterclockwise circuit of length~$d$.
\end{prop}
\begin{proof}
The distributive lattice structure immediately follows from Lemma~\ref{lem:structure}.
Characterizing the cycle-pushings corresponding to covering relations can be done by a study very similar to the one in~\cite{Felsner:lattice} (for classical orientations), therefore we provide only a sketch here. 
Given a $d/(d-2)$ orientation $O$, a  path $P=(v_0,\ldots,v_r)$ is called \emph{directed} if for all $i$ in $\{0,...,r-1\}$ the arc from $v_i$ to $v_{i+1}$ has a positive value. 
A cycle $C$ is said to have a \emph{chordal path} in $O$ if there exists a directed
path starting and ending on $C$ having all its edges inside $C$.
 The cycle $C$ is called \emph{rigid} if it is simple and  has no chordal path
in any $d/(d-2)$-orientation. One easily checks that the pushing of a rigid (counterclockwise) cycle $C$
is a covering relation, that is, can not be realized as a sequence of pushings of some cycles $C_1,\ldots,C_r$
(the case $r=2$ is easy, $C_1$ and $C_2$ must share a continuous portion $P$ such that $C=(C_1\cup C_2)\backslash P$,
and $P$ is a chordal path of $C$; the case $r>2$ is done similarly by grouping $C_1,\ldots,C_r$ as
$C_1\cup(C_2\cup\ldots\cup C_r)$). Let $C$ be a cycle of length $d$ in $G$.
The Euler relation easily implies that, in any $d/(d-2)$-orientation of $G$,
each arc $(v,e)$ inside $C$ with its origin $v$ on $C$ 
has value $0$, hence $C$ is a rigid cycle, so that the pushing of $C$ (when $C$ is counterclockwise)
is a covering relation.
Now consider a counterclockwise circuit $C$ of length length $r>d$ in a $d/(d-2)$-orientation of $G$.
Again by the Euler relation, there is
at least one arc $a=(v,e)$ inside $C$ such that $v$ is on $C$ and $\Omega(a)>0$. Let $i$
be one of the colors of $a$ in the associated Schnyder decomposition, and let $P$ be the path of color $i$ starting
from $a$. Since $P$ ends on one of the external vertices, it has to hit the boundary of $C$, thereby forming a chordal path.
Hence the pushing of $C$ can be realized as two successive cycle-pushings (sharing $P$), so this is
not a covering relation.
\end{proof}

\fig{width=12cm}{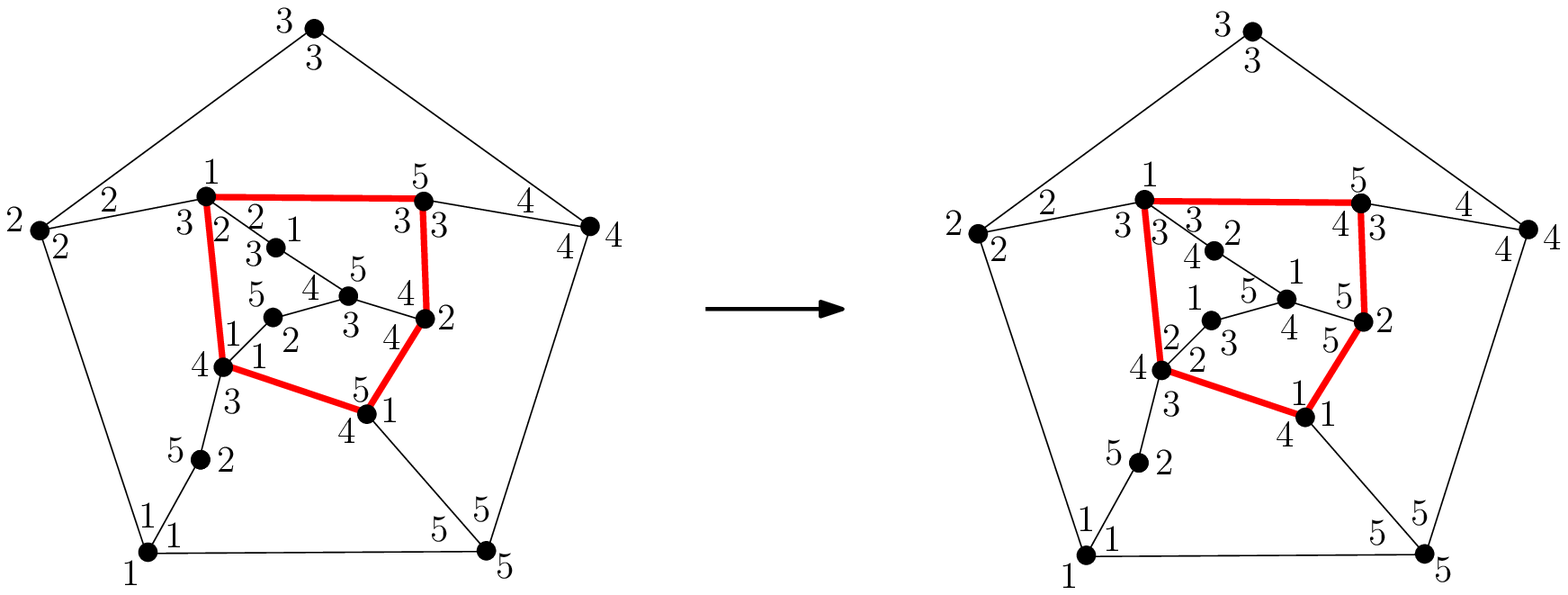}{A covering cycle-pushing seen on clockwise labellings.}

Cycle-pushings have a nice formulation in terms of clockwise labellings.
In this context, an \emph{admissible cycle} is a cycle $C$  such that, for each counterclockwise arc
$a$ of $C$, the colors of the corners separated by $a$ around the origin of $a$ are distinct. The covering relation is the \emph{pushing} of admissible cycles of length $d$, which means the incrementation by $1$ (modulo $d$) of the color of every corner inside the cycle.
Figure~\ref{fig:lattice} illustrates a (covering) cycle-pushing in terms of clockwise-labellings.

\medskip

\section{Dual of Schnyder decompositions}\label{sec:duality}
In this section, we explore the structures obtained as ``dual'' of Schnyder decompositions. Recall that the \emph{dual} $G^*$ of a plane graph $G$ is the map obtained by drawing a vertex $v_f$ of~$G^*$ in each face $f$ of~$G$ and drawing an edge $e^*=\{v_f,v_g\}$ of~$G^*$ across each edge $e$ of~$G$ separating the faces $f$ and $g$ (this process is well-defined and unique up to the choice of the infinite face of~$G^*$, which can correspond to any of the external vertices of~$G$). An example is given in Figure~\ref{fig:dual-Schnyder}(a). The vertex dual to the external face of $G$ is called the \emph{root-vertex}
of $G^*$ and is denoted $v^*$. The edges $e$ and $e^*$ are called \emph{primal} and \emph{dual} respectively. The \emph{dual of a corner} of $G$ is the corner of $G^*$ which faces it. Observe that the degree of the vertex~$v_f$ is equal to the degree of the face~$f$, hence the dual of a $d$-angulation is a $d$-regular plane graph. Moreover, the $d$-angulation has girth $d$ if and only if its dual has mincut $d$.

From now on $G^*$ is a $d$-regular plane graph rooted at a vertex $v^*$. The edges $e_1^*,\ldots,e_d^*$ in counterclockwise order around  $v^*$ are called \emph{root-edges}, and the faces $f_1^*,\ldots,f_d^*$ in counterclockwise order ($e_i^*$ being between $f_i^*$ and $f_{i+1}^*$) are called the \emph{root-faces}.

\begin{Def}
A \emph{regular labelling} of~$G^*$ is the assignment to each corner of a color in $[d]$ such that
\begin{enumerate}
\item[(i)] The colors $1,2,\ldots,d$ appear in clockwise order around each non-root
vertex of~$G$ and appear in counterclockwise order around the root-vertex.
\item[(ii)] For all $i$ in $[d]$, the corners incident to the root-face $f_i^*$ have color $i$.
\item[(iii)] In clockwise order around each non-root face,  there is exactly one corner having a larger color than the next corner.
\end{enumerate}
\end{Def}

Note that this is exactly the dual definition of clockwise labellings on $d$-angulations.
Consequently we have:

\begin{lem}
A vertex-rooted $d$-regular plane graph $G^*$ admits a regular labelling if and only if $G^*$
has mincut $d$. In that case, regular labellings of $G^*$ are in bijection (by giving to a corner the same color as its dual) with clockwise labellings of $G$.
\end{lem}
A regular labelling and the corresponding clockwise labelling are said to be \emph{dual} to each other. We now define a structure which is dual to Schnyder decompositions.

\begin{Def}\label{def:regular_dec}
A \emph{regular decomposition} of~$G^*$ is a
covering of the edges of $G^*$ by $d$ spanning
trees $T_1^*,\ldots,T_d^*$  (one tree for each color $i\in [d]$), each oriented toward $v^*$, and such that
\begin{enumerate}
\item[(i)] Each edge $e^*$ not incident to $v^*$ appears in two of the trees $T_i^*$
and $T_j^*$, and the directions of $e^*$  in $T_i^*$ and $T_j^*$ are opposite.
 \item[(ii)] For $i\in[d]$, the root-edge $e_i^*$ appears only in $T_i^*$.
\item[(iii)] Around any internal vertex $v$, the outgoing edges $e_1,\ldots,e_d$ leading $v$ to its parent in $T^*_1,\ldots,T_d^*$ appear in clockwise order around $v$.
\end{enumerate}
We denote by $(T_1^*,\ldots,T_d^*)$ the decomposition.
\end{Def}

Figure~\ref{fig:dual-Schnyder}(d) represent a regular decomposition of a 5-regular graph $G^*$.
As we have done with Schnyder decompositions, we shall talk about the \emph{colors} of edges and arcs of $G^*$.

\begin{Remark} \label{rk:dual}
Property (iii)  implies that a directed path of color $i+1$ cannot cross a directed path of color $i$ from right to left.
\end{Remark}

\fig{width=\linewidth}{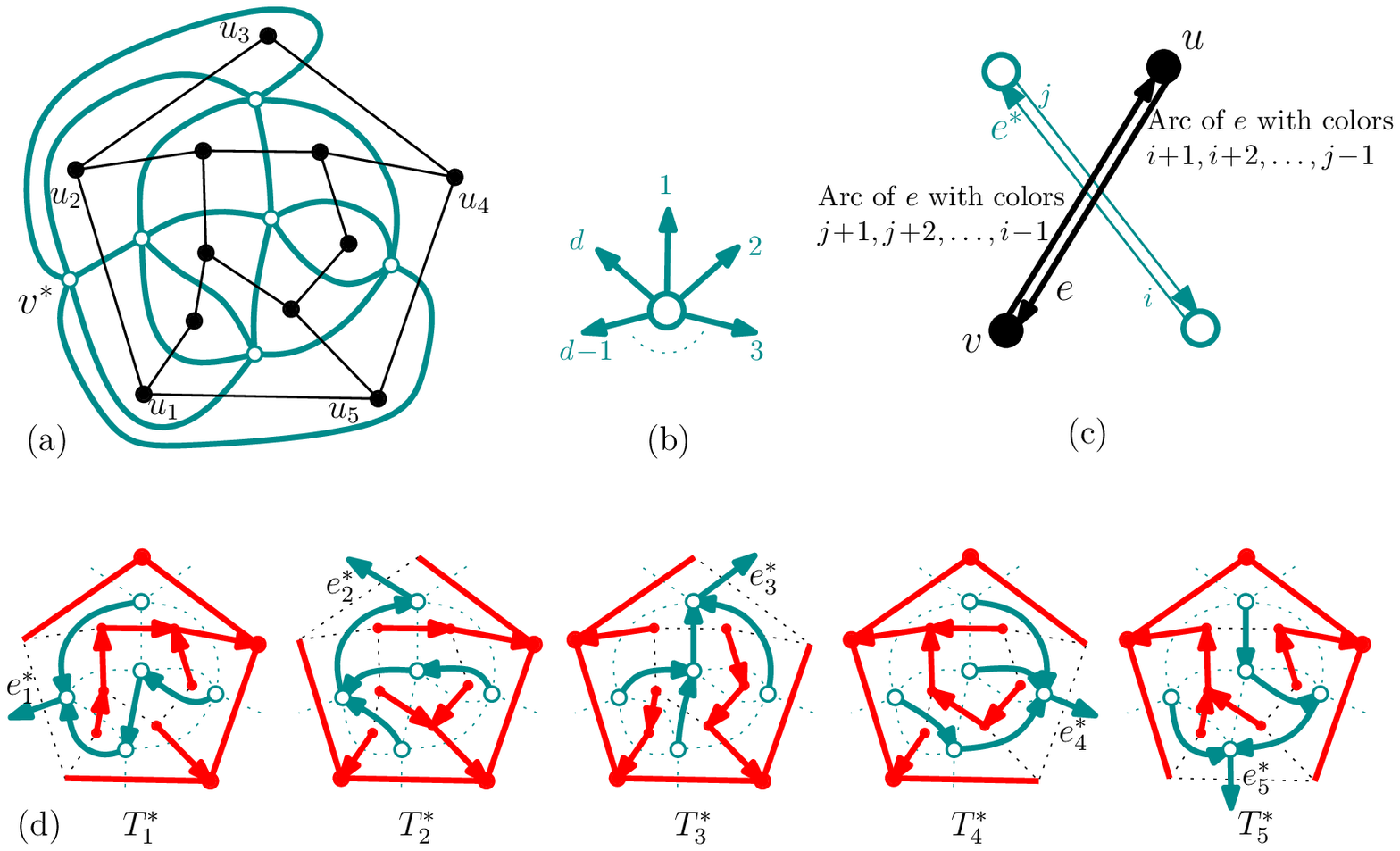}{(a) Dual of a pentagulation. (b) Rule (iii) for the outgoing edges around a non-root vertex. (c) Correspondence between the arc-colors of an edge $e\in G$ and of its dual $e^*\in G^*$ through the mapping $\chi$. (d) A regular decomposition  $(T_1^*,\ldots,T_d^*)=\chi(F_1,\ldots,F_d)$ of $G^*$.}

We now establish the bijective correspondence between regular labellings and regular decompositions of $G^*$.
For a regular labelling $L^*$ of $G^*$, we define $\xi(L^*)=(T_1^*,\ldots,T_d^*)$, where $T_i^*$ is the oriented subgraph of $G^*$ made of every arc $(u,e)$ such that $u\neq v^*$ and the color of the corner preceding $e$ in clockwise order around $u$ is $i$. By duality, $L^*$ corresponds to a clockwise labelling $L$ of $G$, which itself corresponds via the bijection $\Phi$ to a Schnyder decomposition $(F_1,\ldots,F_d)$ of $G$; we then define $\chi(F_1,\ldots,F_d)=(T_1^*,\ldots,T_d^*)$.

\begin{lem}\label{lem:alternative_def}
Let $S=(F_1,\ldots,F_d)$  be a Schnyder decomposition of $G$ and let $R=(T_1^*,\ldots,T_d^*)=\chi(F_1,\ldots,F_d)$. For each internal edge $e$ of $G$, the missing colors of $e$ on $S$ are the colors of the dual edge $e^*$ on $R$, and the orientations obey the rule represented in Figure~\ref{fig:dual-Schnyder}(c).  Hence, if we denote by $T_i$ the spanning tree of $G$ obtained from the forest $F_i$ by adding every external edge except $\{v_i,v_{i+1}\}$, then $T_i^*$ is the spanning tree of $G^*$ which is the \emph{complemented dual} of $T_i$: it is made of all the edges of $G^*$ which are dual of the edges of $G$ not in $T_i$.
Lastly, the spanning tree $T_i^*$ is oriented toward the root-vertex $v_i^*$.
\end{lem}

Lemma \ref{lem:alternative_def} gives a more direct way to define the mapping $\chi$: the subgraphs $(T_1^*,\ldots,T_d^*)=\chi(F_1,\ldots,F_d)$ can be defined as the complemented duals of
$T_1,\ldots,T_d$ oriented toward $v^*$; see Figure~\ref{fig:dual-Schnyder}(d).

\begin{proof}
The first assertion is clear from the definition of the mappings $\Phi$ and $\xi$. From this, it follows that $T_i^*$ is made of all the edges of $G^*$ which are dual of edge not in $T_i$. It is well-known that the complemented dual of a spanning tree is a spanning tree, hence $T_i^*$ is a spanning tree. Lastly, since every vertex of $G^*$ except $v^*$ has one outgoing edge of color $i$, the tree $T_i^*$ is oriented toward $v^*$.
\end{proof}

\begin{thm}\label{thm:equiv-dual}
The mapping $\xi$ is a bijection between regular labellings and regular decompositions of $G^*$. Consequently the mapping $\chi$ is a bijection between the Schnyder decompositions of $G$ and the regular decompositions of $G^*$.
\end{thm}

\begin{proof}
We first show that the image $(T_1^*,\ldots,T_d^*)=\xi(L^*)$ of any regular labelling is a regular decomposition. By Lemma \ref{lem:alternative_def},  for all $i$ in $[d]$, $T_i^*$ is a spanning tree of $G^*$ oriented toward $v^*$ and arriving to $v^*$ via the edge $e_i^*$, 
and each edge non-incident to $v^*$ has two colors in opposite directions. Hence, Properties (i) and (ii) of regular decompositions hold. 
Moreover, it is clear from the definition of $\xi$ that Property~(iii) of 
 regular decompositions holds. Thus, $(T_1^*,\ldots,T_d^*)$ is a regular decomposition.

We now prove that $\xi$ is bijective. The injectivity holds since the regular labelling $L^*$ is clearly recovered from $(T_1^*,\ldots,T_d^*)$ by giving the color $i\in[d]$ to each corner of $G^*$ preceding an outgoing arc of color $i$ in clockwise order around a non-root vertex, and giving the color $i$ also to the corner incident to $v^*$ in the root-face $f_i^*$. To prove that $\xi$ 
is surjective we must show that applying this rule to any regular decomposition $(T_1^*,\ldots,T_d^*)$ gives a regular labelling. For this purpose, we use an alternative characterization of regular labellings (which we state only as a sufficient condition):

\noindent{\bf Claim.}
If $L^*$ is a coloring in $[d]$ of the corners of $G^*$ satisfying the properties below, then $L^*$ is a regular labelling:
\begin{enumerate}
\item[(i')] The colors $1,2,\ldots,d$ appear in clockwise order around each non-root
vertex of~$G$ and appear in counterclockwise order around the root-vertex.
\item[(ii')] For each non-root edge $e=\{u,v\}$,  the corners preceding $e$ in clockwise order around
 $u$ and $v$ respectively have distinct colors.
\item[(iii')]  For each root-edge $e_i^*=\{v^*,v_i^*\}$, the corners
before and after $e_i^*$ in clockwise order around $v^*$ have color $i+1$ and $i$ respectively, and the corners
before and after $e_i^*$ in clockwise order around $v_i^*$ have color $i$ and $i+1$ respectively.
\item[(iv')] No non-root face has all its corners of the same color.
\end{enumerate}

\medskip

\noindent\emph{Proof of the claim.} We suppose that $L^*$ satisfies (i'), (ii'), (iii'), (iv') and have to prove that it satisfies Properties (i), (ii) and (iii) of regular labellings; in fact
we just have to prove it satisfies (ii) and (iii), since (i) is the same as (i').
For this purpose it is convenient to define the \emph{jump along an arc}
$a=(u,e)$ of $G^*$: the jump along $a$
is the quantity in $\{0,\ldots,d-1\}$ equal to $j-i$ modulo $d$, where $i$ and $j$ are the colors of the corners
on the right of $a$ at the origin and at the end of $a$ respectively.
Property~(ii') implies that a jump is never $d-1$. Hence the jumps along the two arcs forming a non-root edge always add up to $d-2$ (because this sum plus 2 is a multiple of $d$).   Property~(iii') implies that the jump along both arcs of a root-edge is $0$. Thus  the sum of jumps along all arcs  is $(d-2) m$, where $m$ is the number of non-root edges of $G^*$. Moreover, by Equation~\eqref{eq:incidence}, $(d-2) m=d\, \mathrm{F}$, where $\mathrm{F}$ is the number of non-root faces of $G^*$. We now compute the sum of jumps as a sum over faces. Note that for any face $f$, the sum of the jumps along the arcs with $f$ on their right is a multiple of $d$. Moreover, Property~(iv') guarantees that this is a positive multiple of $d$ when $f$ is a non-root face. Since the sum of all jumps along arcs is $d\, \mathrm{F}$, we conclude that the sum of jumps is $d$ for any non-root face (implying Property (iii)) and 0 for any root-face (implying Property (ii) since the color of one of the corners of $f_i$ is $i$ by (iii')).
\qedclaim

Given the Claim, it is now easy to check that the coloring $L^*$ obtained from a regular decomposition $S=(T_1^*,\ldots,T_d^*)$ is a regular labelling.
Since $S$ satisfies Property (iii) of regular decompositions, the coloring $L^*$ 
satisfies Property~(i)=(i') of regular labellings.
Since $S$ satisfies Property (i) (resp. (ii)) of regular decompositions, $L^*$ satisfies Property (ii') (resp. (iii')) of the Claim. Lastly, $L^*$ satisfies Property (iv') of the Claim since if a face $f\notin\{f_1^*,\ldots,f_d^*\}$ had all its corner of the same color~$i$, then the whole contour of $f$ would belong to $T_i^*$, contradicting the fact that $T_i^*$ is a tree. Hence, $L^*$ is a regular labelling, and $\xi$ is surjective. Thus, $\xi$ and $\chi$ are bijections.
\end{proof}

\medskip
\section{Even Schnyder decompositions and their duals}\label{sec:even}
In this section we focus on even values of~$d$ and study a special class of Schnyder decompositions and clockwise labellings (and their duals), which are called \emph{even}. 

\subsection{Even Schnyder decompositions and even clockwise labellings}~\\
Let $d=2p$ be an even integer greater or equal to 4. A $d/(d-2)$-orientation is called \emph{even} if the value of every arc is even. Recall that Theorem~\ref{thm:exists} grants the existence of a $p/(p-1)$-orientation for any $2p$-angulation of girth $2p$. Moreover, $p/(p-1)$-orientations clearly identify with even $d/(d-2)$-orientations (by multiplying the value of every arc). The Schnyder decompositions and clockwise labellings associated to even $d/(d-2)$-orientations (by the bijections $\Phi$ and $\Psi$ defined in Section~\ref{sec:equiv-definitions}) are called respectively \emph{even Schnyder decompositions} and \emph{even clockwise labellings}.\\

In the following we consider a $2p$-angulation $G$ of girth $2p$ with external vertices denoted $u_1,\ldots,u_{2p}$ in clockwise order around the external face. The $2p$-angulation $G$ is bipartite (since faces have even length and generate all cycles) hence its vertices can be properly colored in black and white. We fix the coloring by requiring the external vertex $u_1$ to be black (so that $u_i$ is black if and only if $i$ is odd). We first characterize even clockwise labellings, an example of which is presented in Figure~\ref{fig:hexangulation}(a).

\begin{lem} \label{lem:characterization-even-labellings}
A clockwise labelling of~$G$ is even if and only if the corners incident to black vertices have odd colors, while corners incident to white vertices have even colors.
\end{lem}

\begin{proof}  By definition, a clockwise labelling is even if and only if the value of every arc is even in the associated $d/(d-2)$-orientation $O=\Psi(L)$. By definition of the bijection $\Psi$, this condition is equivalent to the fact that $L$ has only even jumps across arcs. Equivalently, \emph{the colors of the corners incident to a common  vertex all have the same parity}. This, in turns, is equivalent to the fact that the parity is odd around black vertices and even around white vertices because the parity changes from a corner to the next around a face (by Property~(i) of clockwise labellings).
\end{proof}

\begin{Remark}
The parity condition ensures that, in an even clockwise labelling, one can
replace every color $i\in [d]$ by $\lfloor (i-1)/2\rfloor$  with no loss of information.
For quadrangulations ($p=2$), this means that the colors are  $0,0,1,1$ around
each face. Such labellings (and extensions of them) for quadrangulations were recently
 studied  by Felsner et al~\cite{FeHuKa}.
The even clockwise labellings were also considered in~\cite{Barriere-Huemer:4-Labelings-quadrangulation} in our form (colors $1,2,3,4$ around
a face) to design a straight-line drawing algorithm for quadrangulations, which 
we will recall in Section~\ref{sec:eq_line}.  
\end{Remark}

We now come to the characterization of even Schnyder decompositions.

\begin{lem}\label{lem:characterization-even-Schnyder}
A Schnyder decomposition of~$G$ is even if and only if the two missing colors of each
internal edge have different parity (equivalently, the edge has as many even as odd colors in $\{1,\ldots,2p\}$). In this case, for all $i\in[p]$ and for each black (resp. white) internal vertex $v$, the edges leading $v$ to its parent in $F_{2i}$ and in $F_{2i-1}$ (resp. in $F_{2i}$
and in $F_{2i+1}$) are the same.
\end{lem}

\begin{proof} Recall that the bijection $\Gamma=\Psi\circ\Phi^{-1}$ from Schnyder decompositions to $d/(d-2)$-orientations is simply the ``color deletion'' mapping. Thus a Schnyder decomposition is even if and only if \emph{the number of colors of every arc of internal edge is even}. Recall from Remark~\ref{rk:immediate-csq-ii}(a) that if $i,j$ are the colors missing from an edge $e$ then the colors $i+1,\ldots,j-1$ are all in one direction and the colors $j+1,\ldots,i-1$ are all in the other direction. Therefore the emphasized property above is equivalent to saying that in the Schnyder decomposition the two missing colors of any internal edge have different parity.

To prove the second statement, recall that, by Lemma~\ref{lem:characterization-even-labellings}, the colors of the corners around a black vertex are odd in an even clockwise labelling. Thus, in the corresponding even Schnyder decomposition, the colors of an arc going out of a black vertex are of the form $2i-1,2i,\ldots,2j-1,2j$ for some $i,j\in[p]$. Similarly, the colors of an arc going out of a white vertex are of the form $2i,2i+1,\ldots,2j,2j+1$ for some $i,j\in[p]$.
\end{proof}

Lemma~\ref{lem:characterization-even-Schnyder} shows that there are redundancies in considering both the odd and even colors of an even Schnyder decomposition.
Let $\Lambda$ be the mapping which associates to an even Schnyder decomposition
$(F_1,\ldots,F_{2p})$ the covering $(F_1',\ldots,F_p')$ of the internal edges edges of $G$ by the forests of even color, that is,  $F_i':=F_{2i}$ for  $i\in [p]$. The forests $F_i'=F_{2i},i=1,2,3$ are represented in Figure~\ref{fig:hexangulation}.

\begin{Def}
A \emph{reduced Schnyder decomposition} of $G$ is a covering of the internal
edges of $G$ by oriented forests $F_1',\ldots,F_p'$ such that
\begin{enumerate}
\item[(i')] Each internal edge $e$ appears in $p-1$ of the forests.
\item[(ii')] For each $i\in [p]$, $F_i'$ spans all  vertices except $u_{2i},u_{2i+1}$; it is made of $2p-2$ trees each containing one of the external vertices $u_j,j\neq 2i,2i+1$, and the tree containing $u_j$ is oriented toward $u_j$ which is considered as its \emph{root}.
\item[(iii')] Around any internal vertex $v$, the edges $e_1',\ldots,e_p'$ leading $v$ to its parent in $F'_1,\ldots,F'_p$ appear in clockwise order around $v$ (some of these edges can be equal). Moreover, if $v$ is a black (resp. white) vertex, the incoming edges of color $i$ are between $e_{i+1}'$ and $e_{i}'$ (resp. between $e_{i}'$ and $e_{i-1}'$)
    in clockwise order
    around $v$ and are distinct from these edges; see Figure~\ref{fig:hexangulation}(c).

\end{enumerate}
\end{Def}

\begin{thm}\label{thm:compact-even-Schnyder}
The mapping $\Lambda$ is a bijection between
even Schnyder decompositions of $G$
and reduced Schnyder decompositions of $G$.
\end{thm}


\fig{width=\linewidth}{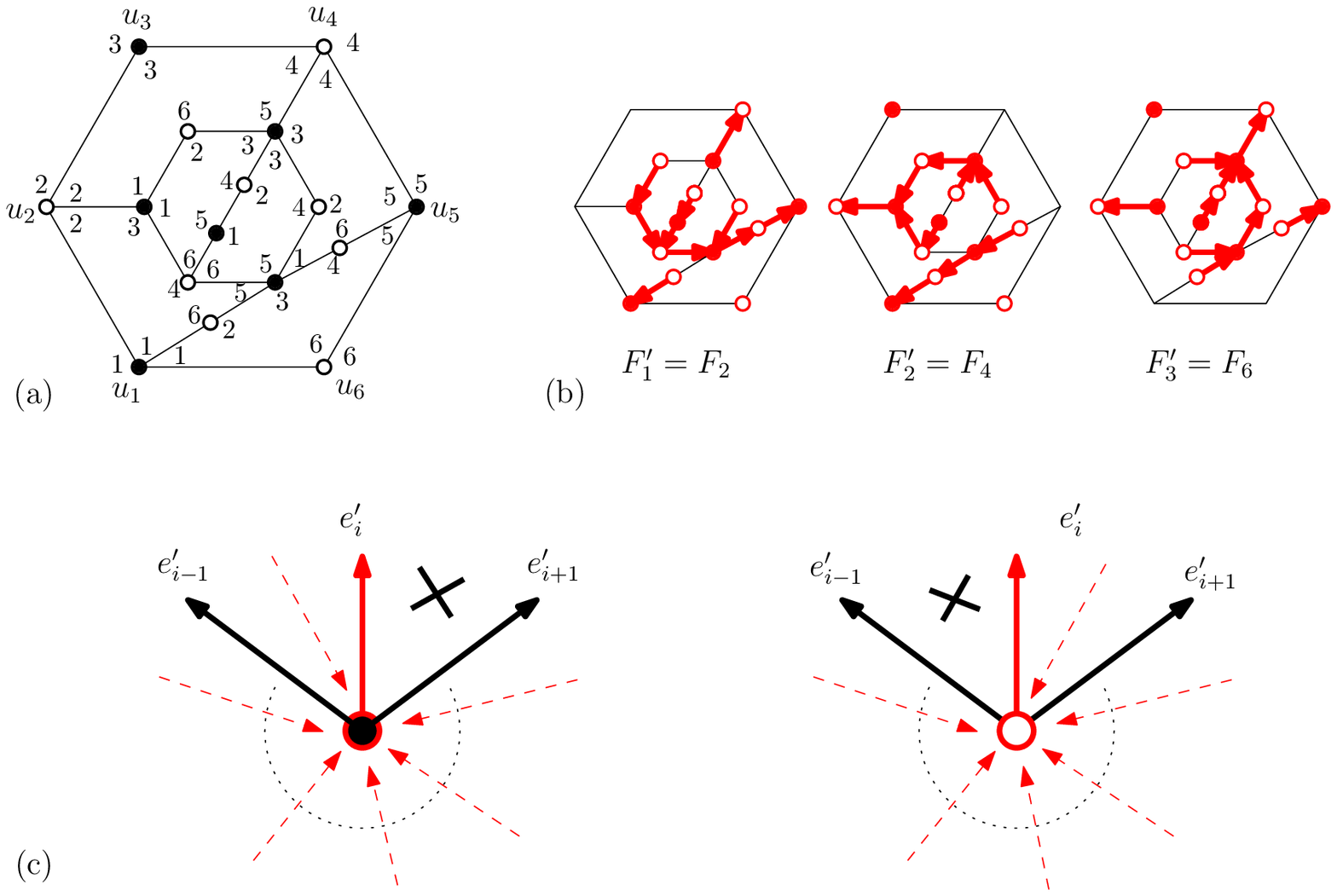}{(a) An even clockwise labelling. (b) A reduced Schnyder decomposition. (c) Property (iii') of reduced Schnyder decompositions.}

As mentioned in the introduction, the case $p=2$ of even Schnyder decompositions already appeared in many places in the literature. Adding to $F_1'$ the external edges
$\{v_4,v_1\}$ and $\{v_1,v_2\}$, and adding $\{v_2,v_3\}$
and $\{v_3,v_4\}$ to $F_2'$, one obtains a pair of non-crossing spanning trees.
Such pairs of trees on quadrangulations
 were first studied in~\cite{Fraysseix:Topological-aspect-orientations}
 and a bijective survey on related structures
appeared recently~\cite{Felsner:Baxter-family}.


\begin{proof}
We first show that if $(F_1',\ldots,F_p')=\Lambda(F_1,\ldots,F_{2p})$ (i.e., $F_i'=F_{2i}$), then Properties (i'), (ii') and (iii') are satisfied.
Properties (i') and (ii') are obvious from Properties (i) and (ii)
of Schnyder decompositions. For Property (iii') we first consider the situation around a black internal vertex $v$. By Lemma~\ref{lem:characterization-even-Schnyder} the edges $e_{2i-1}$ and $e_{2i}=e_i'$ leading $v$ to its parent in $F_{2i-1}$ and $F_{2i}$ respectively are equal, hence Property (iii) of Schnyder
decompositions immediately implies Property (iii') for black vertices. The proof  for white vertices is similar.

We now prove that $\Lambda$ is bijective. Injectivity is clear since Lemma~\ref{lem:characterization-even-Schnyder} ensures that the forests of odd colors
 $(F_1,\ldots,F_{2p})$ can be recovered from the even ones:
\emph{starting from $(F_1',\ldots,F_p')$, one defines $F_{2i}=F_i'$ and then, for all black (resp. white) internal vertex $v$, one gives color $2i-1$ (resp. $2i+1$) to the arc leading from $v$ to its parent in the forest $F_{2i}$}.
To prove surjectivity we must show that applying the emphasized rule to a reduced 
Schnyder decomposition 
 $(F_1',\ldots,F_p')$ always produces an even Schnyder decomposition.
Properties (i), (ii) and (iii) of Schnyder decompositions clearly hold, as well as the characterization of even Schnyder decompositions given in Lemma~\ref{lem:characterization-even-Schnyder}.
The only non-trivial point is to prove that the subgraphs $F_{2i-1},i\in [p]$ are forests oriented toward the external vertices.
If, for $i\in[p]$ the subgraph $F_{2i-1}$ is not a forest, then there is a simple directed cycle $C$ of color $2i-1$ with only internal vertices (since each internal vertex has exactly one outgoing edge of color $2i-1$). If $C$ is directed clockwise, we consider a vertex $v_0$ on $C$ and the edge of color $2i$ going out of $v$. The order of colors in clockwise order around vertices implies that the other end $v_1$ of this edge is either on $C$ or inside $C$. Hence, $v_1$ is an internal vertex and we can consider the edge of color $2i$ going out of $v_1$.
Again the order of colors around vertices implies that the other end $v_2$ of this edge is either on $C$ or inside $C$. Hence, continuing the process we find a directed cycle of color $2i$, which contradicts the fact that $F_{2i}=F_{i}'$ is a forest. Similarly, if we assume that the cycle $C$ is counterclockwise, we obtain a cycle of color $2i-2$ and reach a contradiction. This shows that there is no directed cycle in the subgraphs $F_{2i-1},i\in [p]$, hence that they are forests oriented toward external vertices. Thus,  $(F_1,\ldots,F_{2p})$ is an even Schnyder decomposition.
\end{proof}


\subsection{Duality on even Schnyder decompositions}~\\
Consider a  $2p$-regular graph $G^*$ of mincut $2p$ rooted at a vertex $v^*$. 
The faces of $G^*$ are said to be \emph{black} or \emph{white} respectively if they are the dual of black or white vertices of the primal graph $G$.
Call \emph{even} the regular decompositions of $G^*$ that are dual
to even Schnyder decompositions.
We first characterize even regular decompositions.

\begin{lem}\label{lem:characterization-even-dual-Schnyder}
A regular decomposition $(T_1^*,\ldots,T_{2p}^*)$ of $G^*$ is even if and only if the two colors of every non-root edge have different parity. Equivalently, the spanning trees $T_2^*,T_4^*,\ldots,T_{2p}^*$ form a partition of the edges of $G^*$ distinct from the root-edges $e_1^*,\ldots,e_{2p-1}^*$ (while $T_1^*,T_3^*,\ldots,T_{2p-1}^*$ form a partition of the edges of $G^*$ distinct from $e_2^*,\ldots,e_{2p}^*$). Moreover, in this case, the arcs having an even (resp. odd) color have a black (resp. white) face on their right.
\end{lem}

\begin{proof}
The first part of Lemma~\ref{lem:characterization-even-dual-Schnyder} is obvious from Lemma~\ref{lem:characterization-even-Schnyder}. To prove that the arcs of even color have a black face on their right, we consider a black vertex $v$ of the primal graph $G$ and an incident arc $a=(v,e)$. By Lemma~\ref{lem:characterization-even-Schnyder}, the  colors of $a$ are of the form $2i-1,2i,\ldots, 2j-1,2j$ for certain integers $i,j\in[p]$. Therefore, by Lemma~\ref{lem:alternative_def},  the arc of the dual edge $e^*$ having the even color (i.e., 
color $2i$) has the black face of $G^*$ corresponding to $v$ on its right.
\end{proof}

We denote by $\Lambda^*$ the mapping which associates to an even regular decomposition $(T_1^*,\ldots,T_{2p}^*)$ the subsequence
$({T_1'}^*,\ldots,{T_{p}'}^*)$ of trees of even color,
${T_i'}^*:=T_{2i}^*$ for all $i$ in $[p]$.

\begin{Def}
A \emph{reduced regular decomposition} of~$G^*$ is a
partition of the edges of $G^*$ distinct from the root-edges $e_1^*,e_3^*\ldots,e_{2p-1}^*$
into $p$ spanning trees ${T_1'}^*,\ldots,{T_p'}^*$
(one tree for each color $i\in [p]$) oriented toward $v^*$ such that
\begin{enumerate}
\item[(i')] Every arc in the oriented trees  ${T_1'}^*,\ldots,{T_p'}^*$ has a black face on its right.
\item[(ii')] The only root-edge in  ${T_i'}^*$ is $e_{2i}^*$.
\item[(iii')] Around any internal vertex $v$, the edges $e_1',\ldots,e_p'$ leading $v$ to its parent in ${T_1'}^*,\ldots,{T_p'}^*$ appear in clockwise order around $v$.
\end{enumerate}
We denote by $({T_1'}^*,\ldots,{T_p'}^*)$ the decomposition.
\end{Def}

\begin{thm}\label{thm:compact-even-dual-Schnyder}
The mapping $\Lambda^*$ establishes a bijection between the even regular decompositions and the reduced regular decompositions of $G^*$.
\end{thm}

\begin{proof} It is clear from Lemma~\ref{lem:characterization-even-dual-Schnyder} that the image $({T_1'}^*,\ldots,{T_{p}'}^*)$  of any even regular decomposition by the mapping $\Lambda^*$ satisfies (i'), (ii'), (iii'). Moreover the mapping $\Lambda^*$ is injective since the odd colors can be recovered from the even ones: \emph{starting from $({T_1'}^*,\ldots,{T_p'}^*)$, ones defines $T_{2i}^*={T_i'}^*$ and then around each vertex $v\neq v^*$ one gives color $2i-1$ to the arc going out of $v$ preceding the arc of color $2i$ going out of $v$.} In order to show that $\Lambda^*$ is surjective, we must show that applying the emphasized rule to $({T_1'}^*,\ldots,{T_p'}^*)$ satisfying (i'), (ii'), (iii') always produces an even regular  decomposition. 
Clearly, Property (i') implies that the incoming and outgoing arcs of $({T_1'}^*,\ldots,{T_p'}^*)$ alternate around any non-root vertex. Thus the coloring $(T_1^*,\ldots,T_{2p}^*)$ obtained is such that every non-root edge has two colors of different parity in opposite direction,
and such that the arc colors in clockwise order around a non-root vertex are $(1,2,\ldots,2p)$. 
Hence, to show that $(T_1^*,\ldots,T_{2p}^*)$ is an even regular decomposition,  
it remains to show that each oriented subgraph $T_{2i-1}^*,i\in [p]$ is a spanning tree oriented toward $v^*$ and contains the root-edge $e_{2i-1}^*$. Suppose that the subgraph $T_{2i-1}^*$ is not a tree. In this case, there is a simple directed cycle $C$ of color $i$ with only non-root vertices (since each non-root vertex has exactly one outgoing edge of color $2i-1$). If $C$ is directed clockwise, we consider a vertex $v_0$ on $C$ and the arc of color $2i$ going out of $C$.
  Since colors are consecutive in clockwise order around non-root vertices, the end $v_1$ of this arc is either on $C$ or inside $C$. Hence, $v_1$ is a non-root vertex and we can consider the arc of color $2i$ going out of $v_1$. Again, since the colors are consecutive in clockwise order
   around non-root vertices, the end $v_2$ of this arc is either on $C$ or inside $C$. Hence, continuing the process we find a directed cycle of color $2i$, which contradicts the fact that $T_{2i}^*={T_i'}^*$ is a tree. Similarly, if the cycle $C$ is counterclockwise, we obtain a cycle of color $2i-2$ and reach a contradiction. Thus the subgraphs $T_{2i-1}^*,i\in [p]$ are spanning trees oriented toward $v^*$. Lastly, we must show that the root-edge $e^*_{2i-1}$ is in the tree $T_{2i-1}^*$. Suppose it is not in $T_{2i-1}^*$.
   The directed path $P$ of color $2i$ from $v_{2i-1}^*$ to $v^*$ goes through $v_{2i}^*$ (since it uses $e_{2i}^*$). We consider the cycle $C$ made of $P$ together with the root-edge $e_{2i-1}^*$. Since colors are consecutive in clockwise order around non-root vertices, the path $P'$ of color $2i-1$ from  $v_{2i}^*$ to $v^*$ starts and stays inside $C$ (its edges are either part of $C$ or inside $C$). Thus, $P'$
   must use the root-edge $e_{2i-1}^*$ to reach $v^*$. Hence, $e_{2i-1}^*$ is in $T_{2i-1}^*$.
This completes the proof that  $(T_1^*,\ldots,T_{2p}^*)$ is an even regular decomposition and that $\Lambda^*$ is a bijection.
\end{proof}

\bigskip


\section{Orthogonal and straight-line drawing of 4-regular plane graphs}\label{sec:drawing}
A \emph{straight-line drawing} of a (planar) graph is a planar drawing where each edge is drawn as a segment. An \emph{orthogonal drawing} is a planar drawing where each edge is represented as a sequence of horizontal and vertical segments. We present and analyze an algorithm for obtaining (in linear time) straight-line and orthogonal drawings of $4$-regular plane graphs of mincut $4$.

In all this section, $G$ denotes a $4$-regular plane graph of mincut $4$ rooted at a vertex $v^*$
(i.e., the dual of $G$ is a quadrangulation without multiple edges), having $n$ vertices,  hence $2n$ edges and $n+2$ faces by the Euler relation.
As in previous sections, we denote by $e_1^*,e_2^*,e_3^*,e_4^*$  and call \emph{root-edges} the edges incident to $v^*$ in counterclockwise order, and we denote by $v_1^*,v_2^*,v_3^*,v_4^*$ the other end of these edges (these vertices are not necessarily distinct). We call \emph{root-faces} the 4 faces incident to $v^*$ and \emph{non-root faces} the $n-2$  other ones.  As shown in the previous section, there exists an even regular  decomposition $(T_1^*,\ldots,T_4^*)$ of~$G$ (with $e_i^*\in T_i^*$) and we now work with this decomposition.
 The spanning trees  $(T_1^*,\ldots,T_4^*)$ satisfy the properties of regular decompositions (Definition~\ref{def:regular_dec}) and the additional property of even Schnyder decomposition given by Lemma~\ref{lem:characterization-even-dual-Schnyder}.
An even regular decomposition $(T_1^*,\ldots,T_4^*)$ is shown in Figure~\ref{fig:tetravalent} (left).\\

\fig{width=\linewidth}{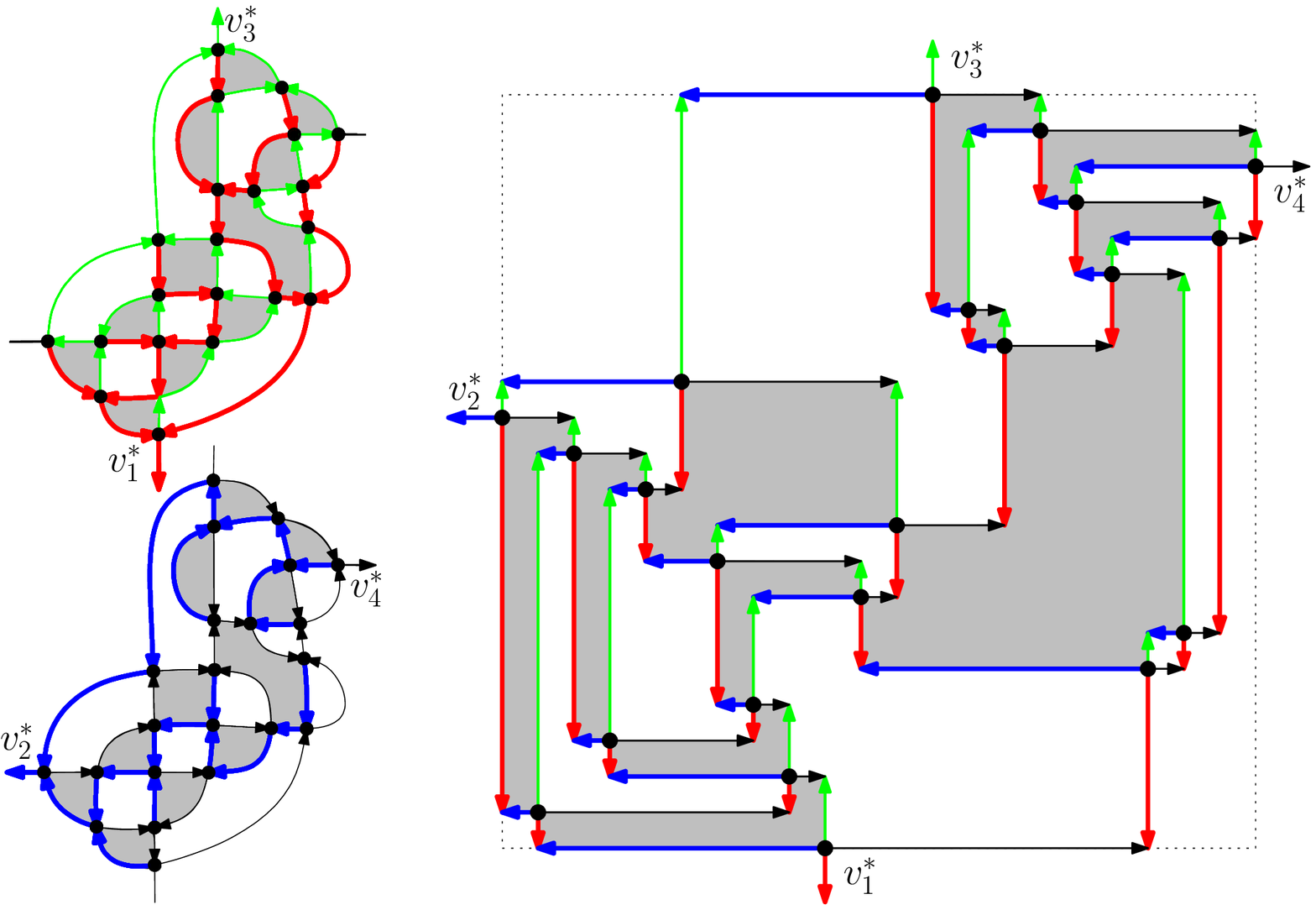}{Left: an even regular decomposition of a 4-regular plane graph $G$ (the spanning trees $T_1^*,T_3^*$ are represented in the upper part and the trees $T_2^*,T_4^*$ are represented in the lower part). Right: the orthogonal drawing of $G$.}

\subsection{Planar drawings using face-counting operations}~\\
For a vertex $v\neq v^*$ and a color $i$ in $\{1,2,3,4\}$, we denote by $P_i(v)$ the directed path of color $i$ from $v$ to $v^*$.
We first establish two easy lemmas about these paths.
\begin{lem}\label{lem:X}
Let $v$ be a non-root vertex and let $i$ be  in $\{1,2,3,4\}$. For any vertex $u\neq v,v^*$ on the path $P_i(v)$, the two arcs going out of $u$ that are not on $P_i(v)$ are on the same side of $P_i(v)$. Moreover, if they are on the left side (resp. right side) of $P_i$, these arcs have color $i+2$ and $i+3$ (resp. $i+1$ and $i+2$). 
\end{lem}
\begin{proof} 
The first assertion comes from the fact that the two colors of any non-root edge of $G^*$ have different parities (and the colors of the arcs out of $u$ are 1,2,3,4  in clockwise order). The second assertion is then obvious.
\end{proof}

\begin{lem}\label{lem:no-intersection}
For all $i$ in $\{1,2,3,4\}$ the paths $P_i(v)$ and $P_{i+2}(v)$ only intersect at $v$ and $v^*$. Equivalently, $P_i(v)\cup P_{i+2}(v)$ is a simple cycle.
\end{lem}
\begin{proof}
Assume the contrary, and consider the first vertex $v'\neq v,v^*$ on the directed path $P_i(v)$ which belongs to $P_{i+2}(v)$. Let $P_i',P_{i+2}'$ be the part of the paths $P_i(v),P_{i+2}(v)$ from $v$ to $v'$. Clearly, $C=P_i'\cup P_{i+2}'$ is a simple cycle not containing $v^*$. For $j\in\{1,2,3,4\}$, let $e_j$ be the edge of color $j$ going out of $v'$. Since the two colors of any non-root edge of $G$ have different parity, neither $e_i$ nor $e_{i+2}$ are on the cycle $C=P_i'\cup P_{i+2}'$. Instead, one of these edges is strictly inside the cycle $C$ (while the other is strictly outside). Let us first suppose that $e_i$ is inside $C$ and  consider the directed path  $P_i(v')$ starting with the edge $e_i$. This path cannot cross $P_i'$ because it would create a cycle of color $i$ and it cannot cross $P_{i+2}'$ because of Lemma \ref{lem:X} (if the path $P_i(v')$ touches $P_{i+2}'$ it bounces back inside $C$). 
Therefore, the directed path $P_i(v')$ is trapped in the cycle $C$ and cannot reach $v^*$, which gives a contradiction. Similarly, the assumption that $e_{i+2}$ is inside $C$ leads to a contradiction.
\end{proof}

For $i$ in $\{1,2,3,4\}$, the cycle $P_i(v)\cup P_{i+2}(v)$ separates two regions of the plane. We denote by  $R_{i,i+2}(v)$ the region containing the root-edge $e^*_{i+1}$.
We also denote by $x(v)$ the number of non-root faces in the region $R_{1,3}(v)$ and
 by $y(v)$ the number of non-root faces in the region $R_{4,2}(v)$.
 In this way one associates to any non-root vertex $v$ the point  $\pp(v)=(x(v),y(v))$ in the grid $\{0,\ldots,n-2\}\times\{0,\ldots,n-2\}$ (shortly called the $(n-2)\times(n-2)$ grid).
 Informally, if the vertices $v_1^*,v_2^*,v_3^*,v_4^*$ are thought as down, left, up, right, then the coordinate $x(v)$ corresponds to the number of faces on the left of the ``vertical line'' $P_1(v)\cup P_3(v)$ and the coordinate $y(v)$ corresponds to the number of faces below the ``horizontal line'' $P_2(v)\cup P_4(v)$. This placement of vertices is represented in~Figure~\ref{fig:tetravalent}. As stated next, it yields both a planar orthogonal drawing
and a planar straight-line drawing.

Before stating the straight-line drawing result, we make the following observation:
if a $4$-regular graph of mincut $4$ has a double edge, then the double
edge must delimit a face (of degree $2$). Denote by $\widetilde{G}$ the simple
graph obtained from $G$ by emptying all faces of degree $2$
(i.e., turning such a double edge into a single edge).

\begin{thm}[straight-line drawing]\label{thm:straightline}
The placement of each non-root vertex $v$  at the point $\pp(v)=(x(v),y(v))$
of the $(n-2)\times(n-2)$ grid gives a planar straight-line drawing of~$\widetilde{G}\setminus v^*$. Moreover the points $\pp(v_1^*),\pp(v_2^*),\pp(v_3^*),\pp(v_4^*)$ are respectively on the down, left, up, and right boundaries of the grid.
\end{thm}

For a vertex $v\neq v^*$ of~$G$ we call \emph{ray in the direction} 1 (resp. 2,3,4) from the point $\pp(v)$ the half-line starting from $\pp(v)$ and going in the negative $y$ direction (resp. negative $x$ direction, positive $y$ direction, positive $x$ direction). For an edge $e=\{u,v\}$ of~$G$ not incident to $v^*$, we denote by $\pp(e)$ the intersection of the ray in direction $i$ from $\pp(u)$ with the ray in direction $j$ from $\pp(v)$, where $i$ is the color of the arc $(u,e)$ and $j$ is the color of the arc $(v,e)$. Observe that one  ray is horizontal while the other is vertical (because $i$ and $j$ have different parity), hence the intersection $\pp(e)$ (if not empty) is a point. If $\pp(e)$ is a point (this is always the case, as we will prove shortly), then we call the union of segments $[\pp(u),\pp(e)]\cup[\pp(v),\pp(e)]$ the \emph{bent-edge} corresponding to~$e$. We say that the bent-edge from $u$ to $v$ is \emph{down-left} (resp. \emph{down-right}, \emph{up-left}, \emph{up-right}) if the vector from $\pp(u)$ to $\pp(e)$ is down (resp. down, up, up) and the vector from $\pp(e)$ to $\pp(v)$ is left (resp. right, left, right). 
We now state the main result of this section.

\begin{thm}[orthogonal drawing]\label{thm:bentline}
For each non-root edge $e=\{u,v\}$ of~$G$, the intersection $\pp(e)$ is a point. Moreover, if one places each non-root vertex $v$ of~$G$ at the points $\pp(v)$ of the $(n-2)\times(n-2)$ grid  and draws the bent-edge $[\pp(u),\pp(e)]\cup[\pp(v),\pp(e)]$ for each non-root edge $e=\{u,v\}$ of~$G$, one obtains a
planar orthogonal drawing of~$G\backslash v^*$ with one bend per edge.
Moreover, the drawing has the following properties:
\begin{enumerate}
\item[(1)] Each line and column of the $(n-2)\times(n-2)$ grid contains exactly one vertex.
\item[(2)] The spanning tree $T_1^*$ (resp. $T_2^*$, $T_3^*$, $T_4^*$) is made of all the arcs  $(u,e)$ such that the vector from $\pp(u)$ to $\pp(e)$ is going down (resp. left, up, right).
\item[(3)] Every non-root face $f$ has two distinct edges  $f_a=\{a,a'\}$ $f_b=\{b,b'\}$ called \emph{special}.
If the face $f$ is black the bent-edges in clockwise direction around $f$ are as follows:  the special bent-edge $\{a,a'\}$ is right-down, the edges from $a'$ to $b$  are left-down, the special bent-edge $\{b,b'\}$ is left-up, the edges from  $b'$ to $a$ are right-up; see Figure~\ref{fig:face-configuration}. The white faces satisfy the same property with \emph{right,down,left,up} replaced by \emph{up,right,down,left}. Moreover, for each black (resp. white) face, one has $x(a)+1=y(b)$ and $y(a')+1=y(b')$ (resp. $x(a')-1=x(b')$ and $y(a)+1=y(b)$).
\end{enumerate}
Adding the root-vertex $v^*$ and its four incident edges $e_1^*$, $e_2^*$,
$e_3^*$, $e_4^*$ requires 3 more rows, 3 more columns,
and 8 additional bends, see Figure~\ref{fig:add_root}.
Overall the planar orthogonal drawing of a 4-regular plane graph of mincut $4$
with $n$ vertices is on the $(n+1)\times (n+1)$ grid and has a total of $2n+4$ bends.
\end{thm}

\begin{figure}[h]
\begin{minipage}{1.1\linewidth}
\hspace{-.8cm}
\begin{minipage}{.65\linewidth}\vspace{.2cm}
\centerline{\includegraphics[scale=.60]{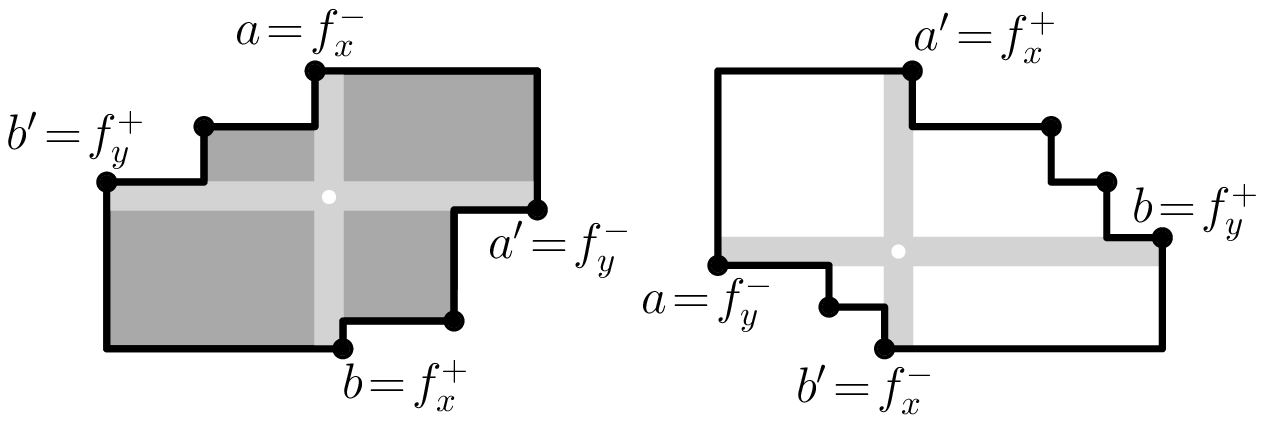}}
\caption{Property (3) for a black face (left) and a white face (right).}\label{fig:face-configuration}
\end{minipage}
\hspace{-1.7cm}\vspace{-.2cm}
\begin{minipage}{.5\linewidth}
\centerline{\includegraphics[scale=.65]{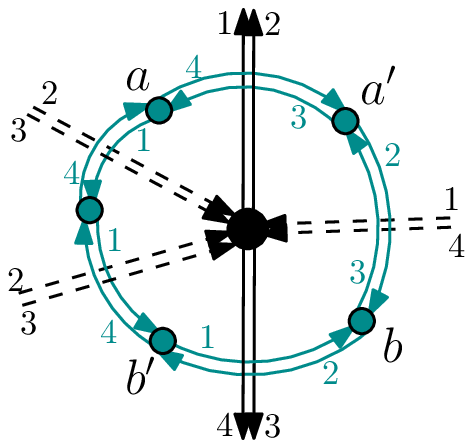}}
\caption{Around a black vertex of a quadrangulation.}\label{fig:proof-drawing-face}
\end{minipage}
\end{minipage}
\end{figure}

\begin{figure}
\begin{center}
\includegraphics[width=8cm]{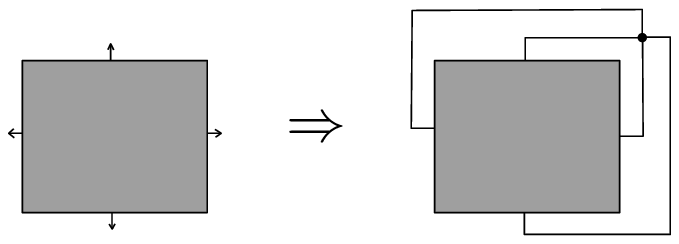}
\end{center}
\caption{Adding the root-vertex to the orthogonal drawing.}
\label{fig:add_root}
\end{figure}



Before starting the proof of Theorems~\ref{thm:straightline} and~\ref{thm:bentline}, we set some terminology. Let $v$ be a non-root vertex and let $i$ be in $\{1,2,3,4\}$. By Lemma \ref{lem:X}, for any vertex $u\neq v,v^*$ on the cycle $C=P_i(v)\cup P_{i+2}(v)$, the two edges incident to $u$ which are not on $C$ are either both strictly in the region $R_{i,i+2}(v)$ or both strictly outside this region.  A vertex $u\neq v,v*$ is said to be \emph{weakly inside} the region $R_{i,i+2}(v)$ if it is either strictly inside this region or on the cycle $P_i(v)\cup P_{i+2}(v)$ with two edges strictly inside this region.
\begin{lem}\label{lem:ordered-by-inclusion}
Let $i$ be a color in $\{1,2,3,4\}$ and let $u,v$ be distinct non-root vertices of $G$.  Then either  $R_{i,i+2}(u)\subsetneq R_{i,i+2}(v)$ or $R_{i,i+2}(v)\subsetneq R_{i,i+2}(u)$. Moreover $R_{i,i+2}(u)\subsetneq R_{i,i+2}(v)$ if and only if  $u$ is weakly inside $R_{i,i+2}(v)$.
\end{lem}

Observe that Lemma~\ref{lem:ordered-by-inclusion} implies that for all $i$ in $\{1,2,3,4\}$ the regions $R_{i,i+2}(v)$ are (strictly) totally ordered by inclusion. In particular, the non-root vertices of $G$ all have distinct $x$ coordinates and distinct $y$ coordinates. Since the $x$ and $y$ coordinates are constrained to be in $\{0,\ldots,n-2\}$, this implies that the vertices are placed according to a permutation: each line and column of the grid contains exactly one vertex\footnote{As mentioned in the introduction, our placement of vertices is closely related to the bijection between plane bipolar orientations and Baxter permutations
in~\cite{Bonichon:Baxter-permutations}.}.
The rest of this section is devoted to the proof of Lemma~\ref{lem:ordered-by-inclusion} and then Theorems~\ref{thm:straightline} and~\ref{thm:bentline}.

\begin{proof}[Proof of Lemma~\ref{lem:ordered-by-inclusion}\\]
\ite We first prove that $R_{i,i+2}(u)\subsetneq R_{i,i+2}(v)$ if and only if $u$ is weakly inside $R_{i,i+2}(v)$.
First suppose that $R_{i,i+2}(u)\subsetneq R_{i,i+2}(v)$. In this case,  $u$ is either strictly inside $R_{i,i+2}(v)$ or on the cycle $P_{i}(v)\cup P_{i+2}(v)$. If $u\in P_{i}(v)$, the edge $e_{i+2}$ of color $i+2$ going out of $u$ is not in $P_{i}(v)\cup P_{i+2}(v)$ (because of Lemma~\ref{lem:X}). Therefore the edge $e_{i+2}$ which belongs to $R_{i,i+2}(u)\subsetneq R_{i,i+2}(v)$ is strictly inside $ R_{i,i+2}(v)$. Thus $u$ is weakly inside $R_{i,i+2}(v)$. A similar proof shows that if $u\in P_{i+2}(v)$, then $u$ is  weakly inside $R_{i,i+2}(v)$. Thus, in all cases $u$ is  weakly inside $R_{i,i+2}(v)$.

We now prove the other direction of the equivalence: we suppose that  $u$ is weakly inside $R_{i,i+2}(v)$ and want to prove that  $R_{i,i+2}(u)\subsetneq R_{i,i+2}(v)$. It suffices to show that the paths  $P_i(u)$ and $P_{i+2}(u)$ have no edge strictly outside of the region $R_{i,i+2}(v)$. We first prove that $P_i(u)$ has no edge strictly outside of $R_{i,i+2}(v)$. Let $e$ be the edge of color $i$ going out of $u$ (i.e. the first edge of the directed path $P_i(u)$). Since $u$ is weakly inside $R_{i,i+2}(v)$ either $e$ belongs to $P_{i}(v)$ (in which case $P_i(u)$ is contained in $P_{i}(v)$) or  $e$ is strictly inside $R_{i,i+2}(v)$. Thus the path  $P_i(u)$ starts inside the region $R_{i,i+2}(v)$. Moreover the path $P_i(u)$ cannot cross the cycle $P_{i}(v)\Cup P_{i+2}(v)$ because, if $P_i(u)$ arrives at a vertex on $P_{i}(v)$ it continues on $P_{i}(v)$, while if it arrives at a vertex on $P_{i+2}(v)$ it bounces back strictly inside $R_{i,i+2}(v)$ by Lemma~\ref{lem:X}. Thus, the path $P_i(u)$ has no edge strictly outside of $R_{i,i+2}(v)$.
Similarly, the path $P_{i+2}(u)$ has no edge strictly outside of the region $R_{i,i+2}(v)$.

\ite We now suppose that the region $R_{i,i+2}(u)$ is not included in  $R_{i,i+2}(v)$ (and want to prove  $R_{i,i+2}(v)\subsetneq R_{i,i+2}(u)$). By the preceding point, this implies that $u$ is not weakly inside $R_{i,i+2}(v)$. In this case,  $u$ is weakly inside the complementary region $R_{i+2,i}(v)$. By the preceding point (applied to color $i+2$) this implies $R_{i+2,i}(u)\subsetneq R_{i+2,i}(v)$. Or equivalently  $R_{i,i+2}(v)\subsetneq R_{i,i+2}(u)$.
\end{proof}


\begin{proof}[Proof of Theorem~\ref{thm:bentline}]
\ite We first prove that for a non-root edge $e=\{u,v\}$, the intersection $\pp(e)$ is a point. It is easy to see that if the arc $(u,e)$ is colored $i$, then the vertex $v$ is weakly inside the region $R_{i-1,i+1}(u)$. By Lemma~\ref{lem:ordered-by-inclusion}, this implies that for $i=1$ (resp. $i=2,3,4$), the point $\pp(v)$ is below
(resp. on the left of, above, on the right of) the point $\pp(u)$. This shows that the intersection $\pp(e)$ is non-empty, hence, a point.\\
\ite We now show that the orthogonal
drawing is planar.
Consider an edge $e=\{u,v\}$ and the segment $[\pp(u),\pp(e)]$ (which is
the embedding
of the arc $(u,e)$). This segment contains no point $\pp(w)$ for $w\neq u$ since every line and column of the grid contains exactly one point. We now suppose for contradiction that the segment $[\pp(u),\pp(e)]$ crosses
the segment $[\pp(u'),\pp(e')]$
for another arc $(u',e')$, with $v'$ the other extremity of $e'$.
  Clearly,  $u'\neq u,v$ and $v\neq u'$ (but the case $v=v'$ is possible). By symmetry between the colors, we can assume that the color of the direction of $e$ from $u$ to $v$ is 1. Thus, the segment $[\pp(u),\pp(e)]$ is vertical with $x(e)=x(u)$ and $y(e)=y(v)<y(u)$. Moreover the segment $[\pp(u'),\pp(e')]$ (which we assume to cross $[\pp(u),\pp(e)]$) is horizontal with $x(e')=x(v')$ and $y(v)<y(e')=y(u')<y(u)$. We now consider the case $x(u)<x(v)$ (the case $x(u)>x(v)$ being symmetric). This means that the color of the direction of $e$ from $v$ to $u$ is 2 (in the other case $x(u)>x(v)$, the color would be~4). The situation is represented in Figure~\ref{fig:proof-drawing2}. Observe that the path $P_1(u)$ is equal to $\{e\}\cup P_1(v)$, while the path $P_2(v)$ is equal to $\{e\}\cup P_2(u)$.

We claim that all the edges of the path $P_4(u)$ lie strictly inside the region $R_{3,1}(u)$. Indeed, the path $P_4(u)$ starts strictly inside $R_{3,1}(u)$ and if it arrives at a non-root vertex of $P_{1}(u)\cup P_{3}(u)$ it bounces back inside  $R_{3,1}(u)$ by Lemma~\ref{lem:X} (indeed it would arrive at the path $P_1(u)$ from its left, and it would arrive at the path $P_3(u)$ from its right).
The same proof shows that all the edges of the path $P_4(v)$ lie strictly inside the region $R_{3,1}(v)$.
Furthermore, $R_{3,1}(v)\subseteq R_{3,1}(u)$ (because $x(u)<x(v)$ and Lemma~\ref{lem:ordered-by-inclusion}), hence the paths $P_4(u)$ and $P_4(v)$ have all their edges strictly inside  $R_{3,1}(u)$.

Let $C$ be the cycle made of $e$ and the parts of the paths $P_4(u)$, $P_4(v)$ between $u$, $v$ and their common ancestor in the tree $T_4^*$ (the region enclosed by $C$ is shaded in Figure~\ref{fig:proof-drawing2}). By the arguments above, we know that all the edges of $C$ except $e$ are strictly inside $R_{3,1}(u)$.
 By Lemma~\ref{lem:ordered-by-inclusion}, the inequality $y(v)<y(u')<y(u)$ implies that $u'$ is weakly inside $R_{2,4}(v)$  and weakly inside $R_{4,2}(u)$. Thus $u'$ is either strictly inside $C$ or on $C$ with two edges strictly inside $C$. Thus, $u'$ has its four incident
 edges strictly inside the region $R_{3,1}(u)$. In particular, $x(u)<x(u')$ by Lemma~\ref{lem:ordered-by-inclusion}.  Hence, if the segment $[\pp(u'),\pp(e')]$ is to cross $[\pp(u),\pp(e)]$, one must have $x(v')<x(u)$. By Lemma~\ref{lem:ordered-by-inclusion}, this implies that $v'$ is weakly inside the region $R_{1,3}(u)$, that is, has its four incident edges in $R_{1,3}(u)$.  Hence the edge $e'=\{u',v'\}$ is both in $R_{1,3}(u)$ (since $e'$ is incident to $v'$) and strictly inside $R_{3,1}(u)$ (since $e'$ is incident to $u'$), which gives a contradiction.

\ite We now examine Properties (1), (2), (3).
Property (1) has already been proved (after Lemma \ref{lem:ordered-by-inclusion}). Property (2) is immediate from the definitions. We now prove Property (3) for a black face $f$ (the case of a white face being symmetric). We first study the colors of the arcs which appear in clockwise direction around $f$. Let $Q=G^*$ be the quadrangulation which is the dual of $G$,  and let  $v$ be the vertex of $Q$ corresponding to the back face $f$. We consider the even clockwise-labelling $L$ of $Q$ corresponding to the even regular decomposition $(T_1^*,\ldots,T_4^*)$ of $G$. In the labelling $L$, the corners in clockwise order around $v$ are partitioned into
two non-empty intervals $I_1,I_3$ such that corners in $I_1$ (resp.  $I_3$)
  are colored $1$ (resp. $3$). Hence the edges in clockwise order around $v$ (see Figure~\ref{fig:face-configuration}) are made of an edge with colors 1,2 oriented away from $v$, a (possibly empty) sequence of edges with colors 4,1 oriented toward $v$, an edge with colors 3,4 oriented away from $v$, and a (possibly empty) sequence of edge with colors 2,3 oriented toward $v$.
Consequently, by Lemma~\ref{lem:alternative_def} on the duality relations of edge colors,
the edges of $G$ in clockwise order around the black face $f$ are made of
an edge with clockwise color 4 and counterclockwise color 3 (right-down bent-edge),
a sequence of edges with clockwise color 2 and counterclockwise color 3 (left-down bent-edges),
an edge with clockwise color 2 and counterclockwise color 1 (left-up bent-edge),
a sequence of edges with clockwise color 4 and counterclockwise color 1 (right-up bent-edges).
Denoting by $e_a=\{a,a'\}$ the edge with clockwise color 4 and counterclockwise color 3 and by $e_b=\{b,b'\}$ the edge with color 2 and counterclockwise color 1, we have proved the first part of Property (3). It remains to prove $x(a)+1=x(b)$ and $y(a')+1=y(b')$.
Observe that the counterclockwise path from $a$ to $b$ around $f$ has color $1$ while the counterclockwise path from $b$ to $a$ around $f$ has color $3$. Therefore the regions $R_{1,3}(a)$ and $R_{1,3}(b)$ only differ by the face $f$: $R_{1,3}(a)\cup \{f\}=R_{1,3}(b)$. Hence, $x(a)+1=x(b)$. Similarly, the clockwise path from $a'$ to $b'$ has color 2 and the clockwise path from $b'$ to $a'$ has color 4. Thus, $R_{4,2}(a')\cup \{f\}=R_{4,2}(b')$ and $x(a')+1=x(b')$.
\end{proof}


\begin{figure}[h]
\begin{minipage}{.52\linewidth}
\centerline{\includegraphics[scale=.5]{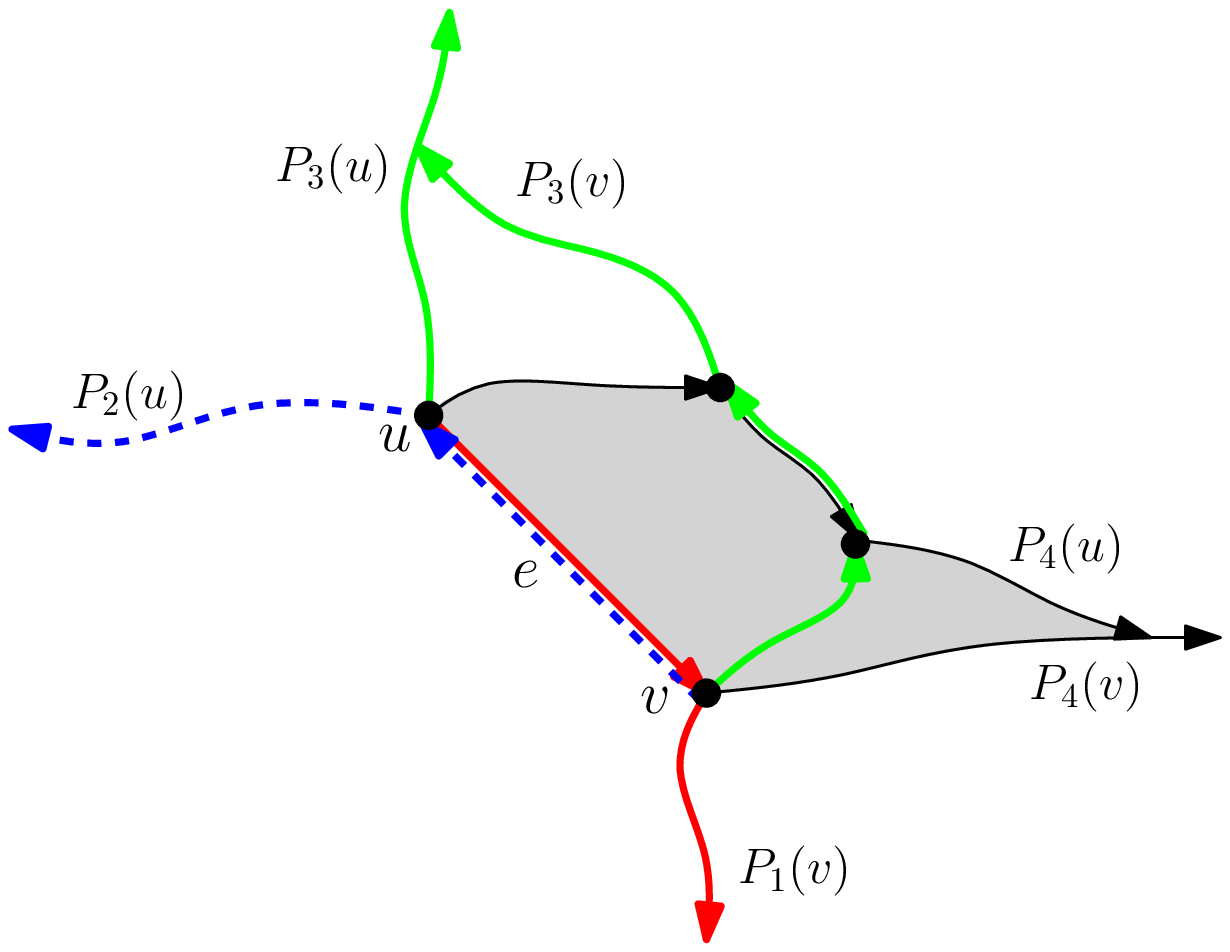}}
\caption{Edge $e=\{u,v\}$ in the proof of Theorem~\ref{thm:bentline}.}\label{fig:proof-drawing2}
\end{minipage}
\begin{minipage}{.47\linewidth}\vspace{1.3cm}
\centerline{\includegraphics[scale=.7]{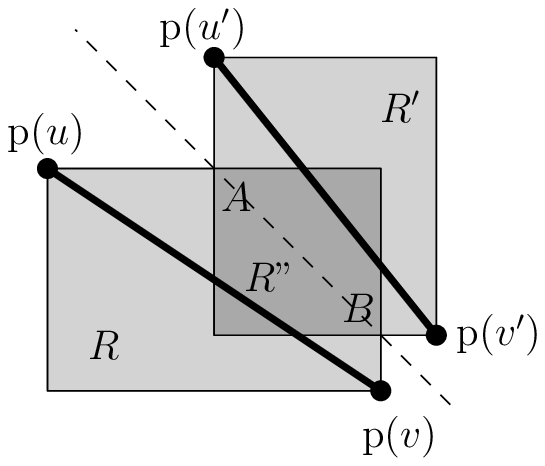}}\vspace{.5cm}
\caption{Rectangles $R,R'$ in the proof of Theorem~\ref{thm:straightline}.}\label{fig:proof-drawing3}
\end{minipage}
\end{figure}

Before embarking on the proof of  Theorem~\ref{thm:straightline}, let us first frame an easy consequence of Property (3) in Theorem~\ref{thm:bentline} (see Figure~\ref{fig:face-configuration}).
\begin{lem}\label{lem:empty-rectangle}
The orthogonal drawing satisfies the following property:
\begin{itemize}
\item[(4)] Each non-root edge $e=\{u,v\}$ is special for exactly one non-root face $f_e$.
Let $R$ be the rectangle with diagonal $[p(u),p(v)]$ (and sides parallel to the axes).
Then $f_e$ is characterized as the unique face of the embedding that contains $R$. Moreover, $u$ and $v$ are the only vertices in $R$ or on the boundary of $R$.
\end{itemize}
\end{lem}

\begin{proof}[Proof of Theorem~\ref{thm:straightline}]
We have to prove that the straight-line drawing is planar. Let $e=\{u,v\}$ be a non-root edge of $G$. Lemma~\ref{lem:empty-rectangle} ensures that no vertex lies on the segment $[\pp(u),\pp(v)]$. We consider another non-root edge $e'=\{u',v'\}$ (with $u,v,u',v'$ distinct) and want to prove that the segments $[\pp(u),\pp(v)]$, $[\pp(u'),\pp(v')]$ do not intersect. We consider the rectangles $R,R'$ with diagonal $[\pp(u),\pp(v)]$ and $[\pp(u'),\pp(v')]$ respectively (and sides parallel to the axes) and their intersection $R''$.
 If $R''=\emptyset$ the segments $[\pp(u),\pp(v)]$, $[\pp(u'),\pp(v')]$ do not intersect, hence we consider the case $R''\neq \emptyset$.
 Note that the boundaries of $R$ and $R'$ can not intersect in $4$ points, otherwise
the representations of the edges $e$ and $e'$ in the orthogonal drawing would have
 to intersect.
 And by Lemma~\ref{lem:empty-rectangle}, the rectangle $R''$ contains none of the points $\pp(u),\pp(v),\pp(u'),\pp(v')$.
 By an easy treatment of the possible cases,
 this implies that $R$ and $R'$ intersect in $2$ points, called
  $A,B$, and the configuration has to be as in Figure~\ref{fig:proof-drawing3} (up to
  a rotation by a multiple of $\pi/2$).
  It clearly appears on Figure~\ref{fig:proof-drawing3} that the points $\pp(u),\pp(v)$ are both on the same side of the line $(A,B)$, while the points $\pp(u'),\pp(v')$ are both on the other side; see Figure~\ref{fig:proof-drawing3}. Thus, the segments $[\pp(u),\pp(v)]$ and $[\pp(u'),\pp(v')]$ are on different sides of $(A,B)$ and do not intersect.
\end{proof}

\subsection{Placing the vertices using equatorial lines}\label{sec:eq_line}
In this subsection, we show that our placement $\pp(v)$ of the vertices can be interpreted, and computed, by considering the so-called \emph{equatorial lines} rather than using face-counting operations. This point of view has the advantage of providing a linear time algorithm for computing the placements of all the vertices. Moreover, it highlights the close relation
between our vertex placement and a straight-line drawing algorithm for simple quadrangulations recently obtained by
Barri\`ere and Huemer~\cite{Barriere-Huemer:4-Labelings-quadrangulation}\footnote{There is also a formulation of the algorithm \cite{Barriere-Huemer:4-Labelings-quadrangulation} in terms of face-counting operations, 
but the formulation with equatorial lines reveals better the relation with our algorithm.}.

\fig{width=\linewidth}{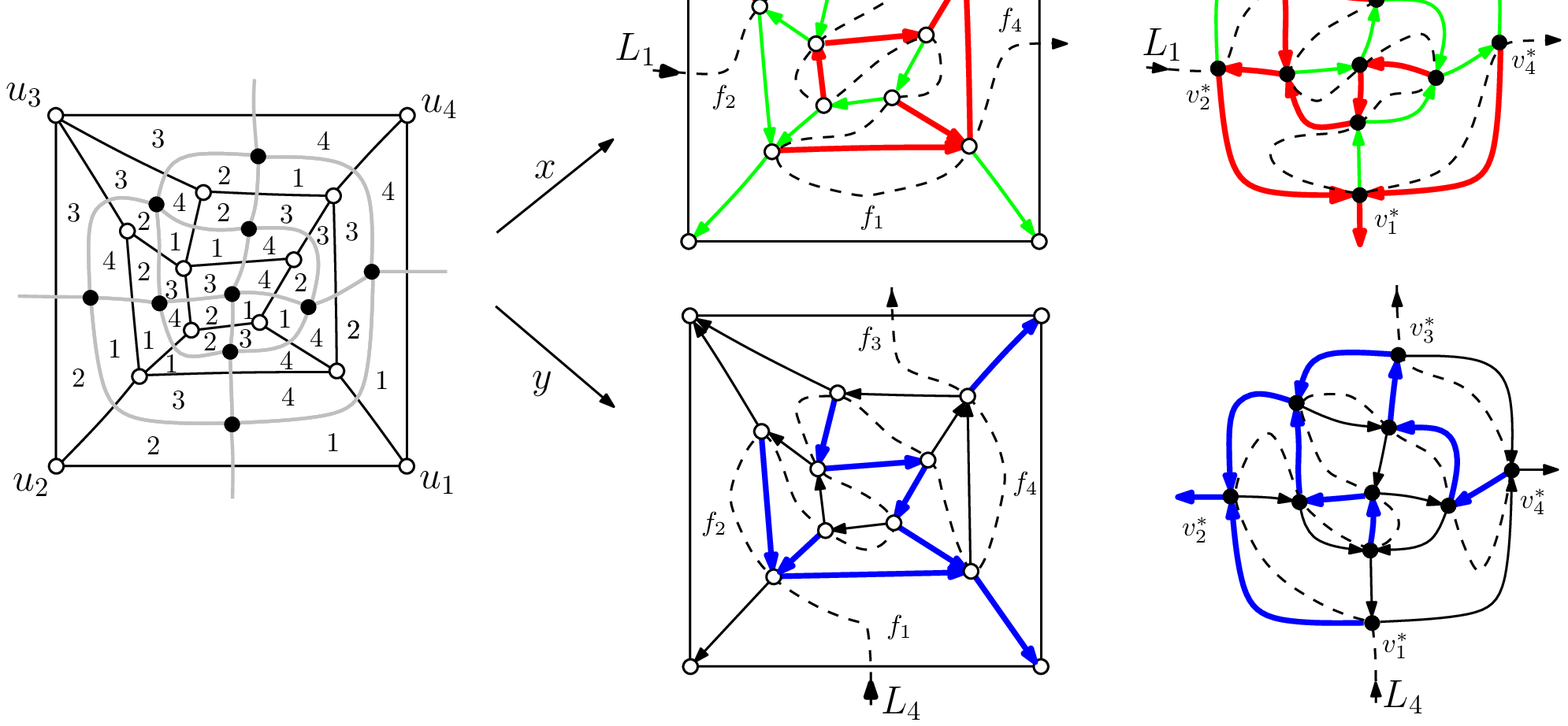}{Left: A 4-regular plane graph $G$ of mincut $4$ superimposed
with its dual quadrangulation $Q$ and endowed with an even clockwise labelling.  Right: the two equatorial lines $L_1$ and $L_4$ drawn on the quadrangulation $Q$ and on the 4-regular graph $G$.}

\fig{width=\linewidth}{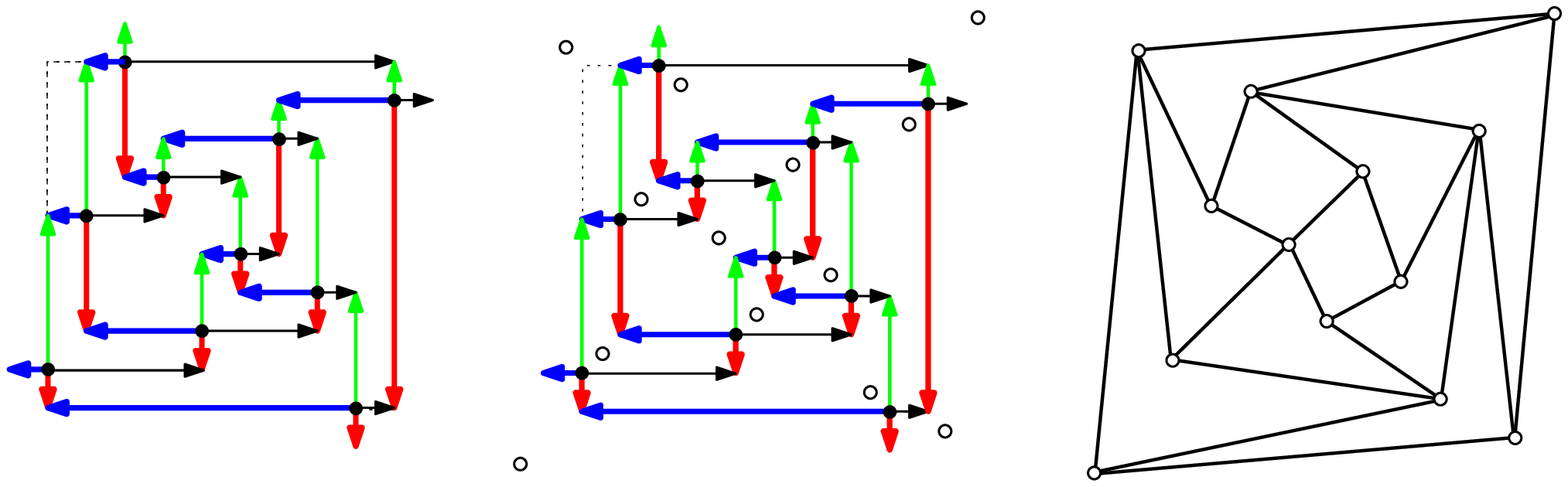}{Orhtogonal drawing of a 4-regular graph $G$, and straight-line drawing for the quadrangulation $Q=G^*$ given by  Barri\`ere and Huemers algorithms~\cite{Barriere-Huemer:4-Labelings-quadrangulation}. In the middle, the two vertex placements are superimposed: each internal vertex of $Q$ is placed at  the ``centre of the cross'' of the corresponding face of~$G$.}

Let $G$ be a vertex-rooted $4$-regular plane graph of mincut $4$, and let $(T_1^*,\ldots,T_4^*)$ be an even regular decomposition. Let $Q=G^*$ be the dual quadrangulation and let $(F_1,\ldots,F_4)=\chi^{-1}(T_1^*,\ldots,T_4^*)$ be the dual Schnyder decomposition. The duality between 
$(T_1^*,\ldots,T_4^*)$ and $(F_1,\ldots,F_4)$ is best seen in terms of clockwise labellings as  illustrated in Figure~\ref{fig:barrieretop} (left). Let $v_1^*,\ldots,v^*_4$ be the neighbors of the root-vertex of $G$, and let $f_1,\ldots,f_4$ be the corresponding internal faces of $Q$.

For $i\in\{1,2,3,4\}$, we denote by $Q_{i}$ the quadrangulation $Q$ with internal edges colored according to the two forests $F_i$ and $F_{i+2}$ of the Schnyder decomposition. An internal corner of $Q_{i}$ is said \emph{bicolored} if the two incident edges are of different colors (one of color $i$ and the other of color $i+2$). 
It is easily seen that each internal vertex and each internal face $f\neq f_{i+1},f_{i+3}$ of $Q$ has exactly two bicolored corners. 
Moreover, it is shown for instance in~\cite{Fraysseix:Topological-aspect-orientations,Felsner:Baxter-family} that by connecting the bicolored corners of each internal face by a segment, ones creates a line $L_i$, called \emph{equatorial line}, starting in the face $f_{i+1}$, ending in the face $f_{i+3}$, passing by each internal vertex and each internal face of $Q$ exactly once, and such that the edges of color $i$ and $i+2$ are respectively on the the left and on the right of $L_i$. The equatorial lines $L_1$ and $L_4$ are represented in Figure~\ref{fig:barrierebottom}. In the algorithm by Barri\`ere and Huemer~\cite{Barriere-Huemer:4-Labelings-quadrangulation}, the $x$-coordinate (resp. $y$-coordinate) of any internal vertex $v$ of $Q$ is equal to its rank along the equatorial line $L_1$ (resp. $L_4$). The straightline drawing obtained by applying the algorithm~\cite{Barriere-Huemer:4-Labelings-quadrangulation} of the quadrangulation $Q$ in Figure~\ref{fig:barrieretop} is represented Figure~\ref{fig:barrierebottom}. It is easily seen that the equatorial lines $L_1,L_4$ (hence all the coordinates) can be computed in linear time. 

In order to establish the link with our algorithm, we need some notations. 
Let $f$ be a non-root face of $G$. By Property (3) in Theorem~\ref{thm:bentline}, the face $f$ has two special edges  $f_a=\{a,a'\}$ and $f_b=\{b,b'\}$. If the face $f$ is black (resp. white), we denote $f_x^-=a$, $f_x^+=b$, $f_y^-=a'$, $f_y^+=b'$ (resp. $f_x^-=b'$, $f_x^+=a'$, $f_y^-=a$, $f_y^+=b$); see Figure~\ref{fig:face-configuration}. We consider the placement $\pp(v)=(x(v),y(v))$ of the non-root vertices of $G$ defined in the previous subsection. By Property (3) in Theorem~\ref{thm:bentline},  $x(f_x^-)+1=x(f_x^+)$ and $y(f_y^-)+1=y(f_y^+)$. We denote $x(f)=x(f_x^-)$ and $y(f)=y(f_y^-)$.  

For $i\in\{1,2,3,4\}$, we denote by $G_{i}$ the $4$-regular plane graph $G$ where non-root edges are colored according to the two trees $T_i,T_{i+2}$ of the regular decomposition. On $G_i$ the equatorial line $L_i$ defined above is a line starting at $v_{i+1}^*$, ending at $v_{i+3}^*$, passing by each non-root face and each non-root vertex of $G$ exactly once, and such that the edges of color $i$ and $i+2$ are respectively on the right and on the left of $L_i$; see Figure~\ref{fig:barrieretop}. 
It easy to check (see Figure~\ref{fig:face-configuration}) that for any face internal $f$ the equatorial line $L_1$ passes consecutively by the vertices $f_x^-$ and $f_x^+$. Moreover, the $x$-coordinates of the vertices $f_x^-$ and $f_x^+$ are also consecutive: $x(f_x^-)+1=x(f_x^+)$. Thus, the $x$-coordinates of the vertices of $G$ in our algorithms are equal to their rank along the equatorial line $L_1$. Similarly, the $y$-coordinates of the  vertices of $G$ in our algorithms are equal to their rank along $L_4$. Hence these coordinates can be computed in linear time. Moreover, if $f$ is a non-root face of $Q$ and $v$ is corresponding vertex of $Q$, the line $L_1$ (resp. $L_4$) passes through $v$ immediately after passing though $f_x^-$ (resp. $f_y^-$). Thus, the $x$-coordinate (resp. $y$-coordinate) of $v$ in the algorithm of Barri\`ere and Huemer~\cite{Barriere-Huemer:4-Labelings-quadrangulation} is equal to $x(f)$ (resp. $y(f)$). Graphically, this means that the placement of vertices of the quadrangulation $Q$ given by~[1] can be superimposed to our straight-line drawing of $G$ in such a way that any vertex $v$ of $Q$ is at the centre of the ``cross'' (see Figure~\ref{fig:face-configuration}) of the corresponding face of $G$.  Figure~\ref{fig:barrierebottom} illustrates this property.

\subsection{Reduction of the grid size}\label{sec:reduction}~\\
In this subsection, we present a way of reducing the grid size while keeping the drawings planar. 
We consider the placement $\pp(v)$ of the non-root vertices of the 4-regular graph $G$. For a non-root face $f$ we adopt the notations $f_x^-,f_x^+,f_y^-,f_y^+,x(f),y(f)$ of Subsection~\ref{sec:eq_line}.
A face $f$ such that the vertices $f_x^-,f_x^+,f_y^-,f_y^+$ are not all distinct is called \emph{non-reducible}.
A face  $f$ such that the vertices $f_x^-,f_x^+,f_y^-,f_y^+$ are all distinct is called \emph{partly reducible} if these are the only vertices around $f$, and \emph{fully reducible} otherwise. The drawing in Figure~\ref{fig:tetravalent-optimized} has 1 partly reducible face and 4 fully reducible faces. A \emph{reduction choice} is a pair $(X,Y)$ with $X,Y\subseteq \{0,\ldots,n-3\}$ such that $X$ contains the $x$ coordinates of the fully reducible faces together with the  $x$ coordinates of a subset of the partly reducible faces, while  $Y$ contains the
 $y$ coordinates of the fully reducible faces together with the  $y$ coordinates of the complementary subset of partly reducible faces.

We now prove that deleting the columns and lines corresponding to any reduction choice gives a \emph{planar} orthogonal drawing, and a \emph{planar} straight-line drawing.
More precisely, given a reduction choice $(X,Y)$, we define new coordinates $\pp'(v)=(x'(v),y'(v))$ for any non-root vertex $v$, where $x'(v)=x(v)-|\{0,\ldots,x(v)-1\}\cap X|$ and $y'(v)=y(v)-|\{0,\ldots,y(v)-1\}\cap Y|$.
Clearly, $x(u)<x(v)$ implies  $x'(u)\leq x'(v)$, and $y(u)<y(v)$ implies  $y'(u)\leq y'(v)$ for any vertices $u,v$. Thus the orthogonal drawing of $G$ with placement $\pp'$ is well defined  (the rays from $\pp'(u)$ and $\pp'(v)$ intersect each other for any edge $e=\{u,v\}$). We call \emph{reduced} the orthogonal and straight-line drawings obtained with the the new placement $\pp'$ of vertices. These drawings are represented in Figure~\ref{fig:tetravalent-optimized}.

\begin{prop}\label{prop:optimization}
For any reduction choice  $(X,Y)$ for $G$, the reduced orthogonal drawing of $G\backslash v^*$ is planar, and every edge has exactly one bend. Moreover, denoting by $\widetilde{G}$ the plane graph obtained
from $G$ by collapsing each face of degree $2$ in $G$, the reduced straight-line drawing of $\widetilde{G}\backslash v^*$ is planar.
\end{prop}

\fig{width=\linewidth}{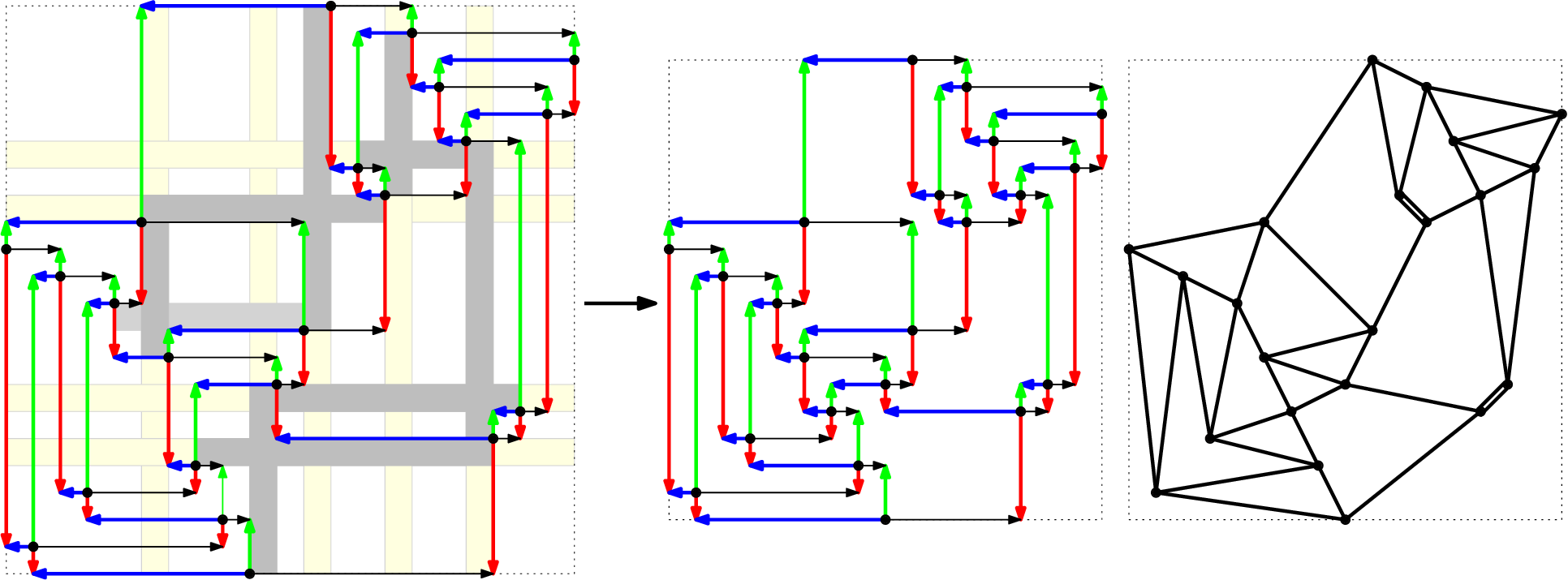}{Reduced orthogonal and straight-line drawings.}



The rest of this subsection is devoted to the proof of Proposition~\ref{prop:optimization}.
\begin{lem}\label{lem:remains-bent}
Let $u,v$ be adjacent non-root vertices of $G$. If $x(u)<x(v)$, then $x'(u)< x'(v)$. Similarly, if $y(u)<y(v)$, then $y'(u)< y'(v)$. In particular the reduced orthogonal drawing has exactly one bend per edge.
\end{lem}

\begin{proof}
We show the property for the $x$ coordinates. Suppose for contradiction that there exists an edge $e=\{u,v\}$ with $x(u)<x(v)$ and $x'(u)=x'(v)$. We can choose $e$ such that the difference $x(v)-x(u)$ is minimal. Clearly  $x'(u)=x'(v)$ means that all the columns between $\pp(u)$ and $\pp(v)$ have been erased: $\{x(u),x(u)+1,\ldots,x(v)-1\}\subseteq X$.
By Lemma~\ref{lem:empty-rectangle}, we know that $e$ is the special edge of a face $f$. Since $x(u)<x(v)$, one has either $u=f_x^-$ or $v=f_x^+$ (see Figure~\ref{fig:face-configuration}).
We first assume that $u=f_x^-$. In this case, $x(f)=x(u)$ is in $X$, thus the face $f$ is either partly or fully reducible. Hence, $v\neq f_x^+$. We consider the non-special edge $e'=\{f_x^+,v'\}$ around $f$. Since $x(u)<x(f_x^+)\leq x(v')-1\leq x(v)-1$ and $\{x(u),\ldots,x(v)-1\}\subseteq X$, one gets  $x'(f_x^+)=x'(v')$ in contradiction with our minimality assumption on $e$. The alternative asumption  $v=f_x^+$ also leads to a contradiction by a similar argument.
\end{proof}

Lemma~\ref{lem:remains-bent} gives the following Corollary illustrated in Figure~\ref{fig:face-reduction}.
\begin{cor}\label{cor:face-config}
The following property holds after deletion of any subset of columns in $X$ and any subset of lines in  $Y$.
\begin{itemize}
\item[(3')] Let $f$ be a non-root face of $G$ and $f_a=\{a,a'\}$, $f_b=\{b,b'\}$ be its special edges. If $f$ is black the bent-edges in clockwise direction are as follows:  the special bent-edge $\{a,a'\}$ is right-down, the edges from $a'$ to $b$  are left-down, the special bent-edge $\{b,b'\}$ is left-up, the edges from  $b'$ to $a$ are right-up.  The white faces satisfy the same property with \emph{right, down, left, up} replaced by \emph{up, right, down, left}.

 Moreover, for black (resp. white) faces, the coordinates $x'',y''$ of the point after the partial reduction satisfy $x''(a)+\eps_x=x''(b)$, where $\eps_x=0$ if the column $x(f)$ has been deleted and 1 otherwise. Similarly $y''(a')+\eps_y=y''(b')$ (resp. $y''(a)-\eps_y=y''(b)$), where $\eps_y=0$ if the line $y(f)$ has been deleted and 1 otherwise.
\end{itemize}
\end{cor}

\fig{scale=.55}{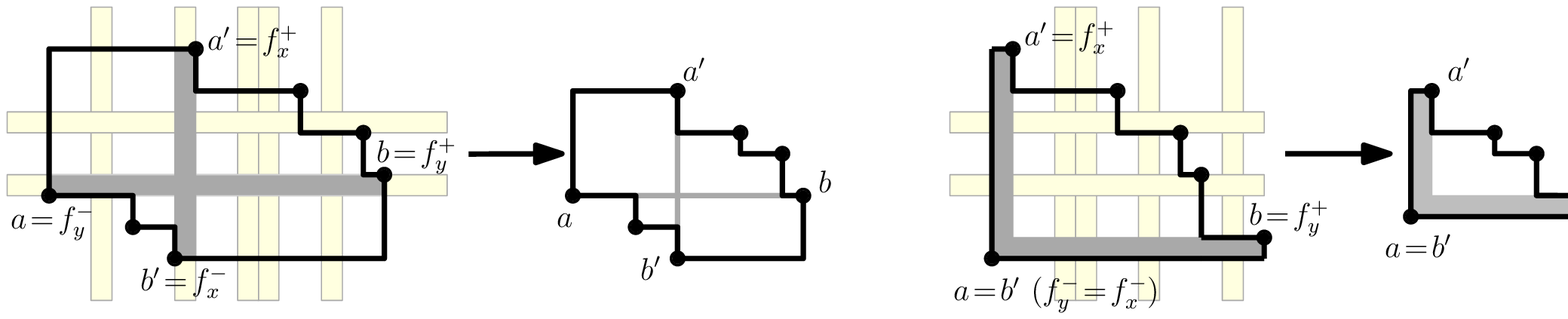}{Effect of reduction on a fully reducible face (left) and a non-reducible face (right).}


\begin{proof}[Proof of Proposition~\ref{prop:optimization}] We first prove the planarity of the orthogonal drawing.
We call \emph{reduction process} the fact of deleting the columns in $X$ and lines in $Y$ one by one (in any order). We want to prove that there is no crossing at any time of the reduction process. Suppose the contrary and consider the first crossing. We suppose by symmetry that it occurs when deleting a column. In this case, it has to be that two vertical segments of the orthogonal drawing are made to superimpose by the column deletion. Just before the collision the two vertical segments had to be part of a common face $f$.
By Corollary~\ref{cor:face-config} the only possible collision is between the vertical rays out of the  vertices $f_x^-$ and $f_x^+$ (indeed around $f$ the vertical segments that were initially strictly to the left of $f_x^-$ remain so, and the vertical segments that were initially strictly to the right of $f_x^+$ remain so, see Figure~\ref{fig:face-reduction}). However, for the vertical rays out of  $f_x^-$ and $f_x^+$ to collide it is necessary that $x(f)\in X$ (so that the rays get to the same vertical line) and that $y(f)\in Y$ (since one of the ray is below $f_y^-$ and the other is above $f_y^+$). Thus, the face $f$ is fully reducible. But in this case it is clear that the  vertical rays out of  $f_x^-$ and $f_x^+$  cannot collide (because either one is strictly below $f_y^-$ or the other is strictly above $f_y^+$, see Figure~\ref{fig:face-reduction}). We reach a contradiction, hence the reduced orthogonal drawing is planar.

The proof that the reduced straight-line drawing is non-crossing is identical to the proof given in Section~\ref{sec:drawing} for the non-reduced straight-line drawing (proof of Theorem~\ref{thm:straightline}). Indeed this proof is entirely based on the non-crossing property of the orthogonal drawing and Property (4) given in Lemma~\ref{lem:empty-rectangle}, which still holds for the reduced orthogonal drawing because of Corollary~\ref{cor:face-config}. 
\end{proof}
\smallskip

\subsection{Analysis of the grid-reduction for a random instance}~\\
In this subsection, we analyze the typical grid size given by our drawing algorithms after the reduction step (presented in Subsection \ref{sec:reduction}) for a large, uniformly random, 4-regular graph endowed with an even regular decomposition. 

A sequence $E_n$ of random events is said to have \emph{exponentially small probability} if $\mathbb{P}(E_n)=O(e^{-cn})$ for some $c>0$. A sequence of real random variables $(X_n)_{n\in\NN}$ is said to be \emph{strongly concentrated} around $\al\, n$ (for a real number $\al\neq 0$) if for all $\eps>0$, the event $\{X_n\notin[\al(1-\eps)n,\al(1+\eps)n]\}$ has exponentially small probability.

Let $n$ be a positive integer. We denote by  $\mR_n$ the set of pairs $(G,R)$ where $G$ is a vertex-rooted $4$-regular plane graph of mincut $4$ with $n$ vertices, and $R$ is an even regular decomposition of $G$. We denote by $R_n$ a uniformly random choice in $\mR_n$. 
We want to study the grid size of the drawing of $R_n$. As reduction-choices for $R_n$ we consider the ones that are \emph{balanced}, that is, such that the numbers of partly reducible faces in the two sets $X$ and $Y$ differ by at most $1$. Thus the grid size is determined by the numbers of number of partially and fully reducible faces in $R_n$. 
 We now state the main result of this section.

\begin{prop}\label{prop:reduc}
The numbers of partly and fully reducible faces of the uniformly random even regular decomposition $R_n$ are strongly concentrated around $n/16$ and $3n/16$ respectively. Consequently, for a balanced reduction choice, the grid-size after reduction is strongly concentrated around $25n/32\times 25n/32$.
\end{prop}

The rest of the subsection is dedicated to the proof of Proposition~\ref{prop:reduc}. In order to use some results from the literature we go to the dual setting. We denote by $\mF_n$ the set of pairs $(Q,F)$ where $Q$ is a quadrangulation with $n$ faces and $F$ is an even Schnyder decomposition of $Q$. For $(Q,F)$ in $\mF_n$, we denote  $(G,R)=(Q,F)^*$ if $G=Q^*$ and $R=\chi(F)$. For an element  $(Q,F)$ in $\mF_n$, we consider the forests $(F_1',F_2')$ of the reduced Schnyder decomposition corresponding to $F$, and we denote $\tau(Q,F)=(T_1',T_2')$ where  $T_1'$ (resp. $T_2'$) is the tree obtained from $F_1'$ (resp. $F_2'$) by adding the external edges $\{u_1,u_2\}$ and  $\{u_1,u_4\}$ (resp. $\{u_3,u_2\}$ and  $\{u_3,u_4\}$). It is easy to rephrase the property for a face of to be reducible in the dual setting; see Figure~\ref{fig:proof-drawing-face} for an illustration of duality.

\begin{lem}\label{lem:reducible-vertex}
Let $(G,R)\in\mR_n$, let $(Q,F)=(G,R)^*$ and let $(T_1',T_2')=\tau(Q,F)$. Let $f$ be a non-root face of the 4-regular graph $G$, let $v$ be the corresponding internal vertex of the quadrangulation $Q$, and let $d_1$ and $d_2$ be the degree of $v$ in the trees $T_1'$ and $T_2'$ respectively. Then, the face $f$ is partly reducible (resp. fully reducible) if and only if $(d_1, d_2)=(2,2)$  (resp.  $d_1\geq 2$, $d_2\geq 2$ and $(d_1, d_2)\neq (2,2)$).
\end{lem}

We now recall a result from~\cite{FuPoSc09}.

\begin{lem}\label{lem:F1F2}
Let $(Q,F)\in\mF_n$ and let $(T_1',T_2')=\tau(Q,F)$. The pair $(Q,F)$ is entirely determined by the plane tree $T_1'$ (considered as a plane tree rooted at the external corner of $u_1$) together with the degree in $T_2'$ of its white vertices.
\end{lem}

Let $(Q,F)\in\mF_n$ and let $(T_1',T_2')=\tau(Q,F)$.  
Let $r+1$ and $s+1$ be the number of black and white vertices respectively.
The quadrangulation $Q$ has $n$ faces and $n+2$ vertices, so that $r+s=n$. 
By definition, the tree $T_1'$ contains every vertex of $Q$ except the black vertex $u_3$. 
Let $\sigma=(v_1,\ldots,v_{n+1})$ be the sequence of vertices discovered  while walking
in clockwise order around $T_1'$ starting from the root. Let $(b_1,\ldots,b_r)$ be the subsequence of black vertices
 and let $(w_1,\ldots,w_{s+1})$ be the subsequence of white vertices in $\sigma$.
 Let $\alpha=(\al_1,\ldots,\al_r)$ be the sequence of degrees of $(b_1,\ldots,b_r)$ in $T_1'$, and 
 let $\beta=(\be_1,\ldots,\be_{s+1})$ and $\ga=(\ga_1,\ldots,\ga_{s+1})$ be the sequence of degrees of $(w_1,\ldots,w_{s+1})$ in $T_1'$ and in $T_2'$.
 It is a classical exercise to show that $T_1'$ can be reconstructed from
 the degree-sequence of $(v_1,\ldots,v_{n+1})$ and similarly $T_1'$ can be reconstructed
 from the two degree-sequences $\alpha$ and $\beta$. 
Hence, by Lemma~\ref{lem:F1F2}, the triple $(\alpha,\beta,\gamma)$ completely encodes $(Q,F)$, hence also $(G,R)=(Q,F)^*$.
Note also that
 \begin{equation}\label{eq:sumIJK}
 \sum_{i=1}^r\al_i~=~\sum_{i=1}^{s+1}\be_i~=~\sum_{i=1}^{s+1}\ga_i~=~n,
 \end{equation}
since the trees $T_1'$ and $T_2'$ both have $n$ edges.

Let $\mT_{r,s}$ be the set of triples $(\alpha,\beta,\gamma)$ of positive integer sequences of respective length $r$, $s+1$ and $s+1$. For an element $T=(\al,\be,\ga)$ in $\mT_{r,s}$, we denote by $\pf(T)$ (resp. $\ff(T)$) the number of indices $i$ in $[s+1]$ such that $(\be_i,\ga_i)=(2,2)$ (resp. $\be_i\geq 2$, $\ga_i\geq 2$ and $(\be_i,\ga_i)\neq (2,2)$). By Lemma \ref{lem:reducible-vertex}, if the triple $(\al,\be,\ga)\in\mT_{r,s}$ encodes a even regular decomposition $(G,R)\in \mR_n$, then $\pf(T)$ and $\ff(T)$ are respectively the number of partly reducible and fully reducible white faces of $(G,R)$. 

We will now study a random variable $T_n$ taking value in the set $\mT_n:=\cup_{r+s=n}T_{r,s}$.
We call \emph{2-geometric} a random variable taking each positive integer value $k$ with probability $2^{-k}$. 
Let $A_{i},B_{i},C_{i}$, $i\in[n]$ be independent 2-geometric random variables. 
We define the random variable $T_n$ taking value in $\mT_n$ by setting $T_n=(\al,\be,\ga)$, where 
 $\al=(A_{1},\ldots,A_{r})$, $\be=(B_{1},\ldots,B_{s+1})$, $\ga=(C_{1},\ldots,C_{s+1})$, and  $r$ is the smallest integer such that $A_{1}+\ldots+A_{r}\geq n$ and $s=n-r$. 

\begin{lem}\label{lem:Anconc}
The random variables $\pf(T_n)$ and $\ff(T_n)$ are  strongly concentrated around $n/32$ and  $3n/32$ respectively.
\end{lem}
\begin{proof}
We first recall the law of large numbers:
for a sequence of independent random variables $(X_1,X_2,\ldots)$ each drawn under the same probability
distribution with exponential tail (i.e., $\mathbb{P}(k)=O(e^{-ck})$ for some $c>0$) and expectation $\mu$,
the random variable $Y_n:=X_1+\cdots+X_n$ is strongly concentrated around $\mu\, n$.
 
We fix $\eps>0$ and define  $m= \lfloor \tfrac{n}{2}(1-\tfrac{\eps}{2})\rfloor$ and $M=\lfloor \tfrac{n}{2}(1+\tfrac{\eps}{2})\rfloor$. Let $r(A)$ be the  smallest integer $r$ such that $A_1+\cdots+A_r\geq n$. The 2-geometric random variables have expectation 2 and exponential tails, so the law of large numbers implies that the event $r(A)\notin[m,M]$ (equivalently, $A_1+\cdots+A_{m-1}\geq n$ or  $A_1+\cdots+A_M<n$) has exponentially small probability. 

For $i\in [n]$, let $X_i$ be 1 if $(B_i,C_i)=(2,2)$ and 0 otherwise, and let $Y_i=X_1+\cdots+X_i$. The random variable $X_i$ has expectation $1/16$, so the law of large numbers implies that the probability that $Y_m<\frac{n}{32}(1-\eps)$ is exponentially small. Similarly, the probability that $Y_M>\frac{n}{32}(1+\eps)$ is exponentially small.
By definition  $\pf(T_n)=Y_{n+1-r(A)}$, hence the event $\{\pf(T_n)\notin[\tfrac{n}{32}(1\!-\!\eps),\tfrac{n}{32}(1\!+\!\eps)]\}$ is contained in the union
$$
 \{n+1-r(A)\notin[m,M]\}~\cup~\{Y_m\leq \tfrac{n}{32}(1\!-\!\eps)\}~\cup~\{Y_M\geq\tfrac{n}{32}(1\!+\!\eps)\},
$$
which has exponentially small probability. Thus $\pf(T_n)$ is strongly concentrated around $n/32$. Similarly, $\ff(T_n)$ is strongly concentrated around $3n/32$.
\end{proof}

We now conclude the proof of Proposition \ref{prop:reduc}. Let $\mE_n\subset \mT_n$ be the set of triples $(\al,\be,\ga)$ encoding
an even regular decomposition  $(G,R)\in\mR_n$.  
Let $(\al,\be,\ga)$ be any triple in $\mE_n$, and let $r$ and $s+1$ be the respective lengths of the sequences~$\al$ and~$\be$. The event $T_n=(\al,\be,\ga)$ occurs if and only if for all $i\in[r]$, $A_i=\al_i$ and for all $i\in[s+1]$, $B_{i}=\be_i$ and $C_{i}=\ga_i$. By \eqref{eq:sumIJK}, this event has probability $8^{-n}$ for any triple $(\al,\be,\ga)$ in $\mE_n$. Thus, by conditioning $T_n$ to belong to $\mE_n$, we obtain a uniformly random element of $\mE_n$. 

Let $W_n$ be the number of partially reducible white faces in the uniformly random element $R_n\in\mR_n$. The random variable $W_n$ is distributed as $\pf(E_n)$, where $E_n$ is a uniformly random element of $\mE_n$. Thus, for any $\eps>0$, 
$$
\begin{array}{lll}
\PP\left(W_n\notin \left[\frac{(1-\eps)n}{32},\frac{(1+\eps)n}{32}\right]\right)&=&
\PP\left(\pf(T_n)\notin \left[\frac{(1-\eps)n}{32},\frac{(1+\eps)n}{32}\right]~|~ T_n\in \mE_n\right)\\
&\leq &\displaystyle \frac{\PP\left(\pf(T_n)\notin \left[\frac{(1-\eps)n}{32},\frac{(1+\eps)n}{32}\right]\right)}{P(T_n\in \mE_n)}.
\end{array}
$$
By Lemma \ref{lem:Anconc}, the probability  $\PP(\pf(T_n)\notin [(1-\eps)n/32,(1+\eps)n/32])$ is exponentially small. 
Moreover, it is shown in~\cite{Felsner:Baxter-asymptotics} that $|E_n|\equiv |\mR_n|\geq \kappa 8^nn^{-4}$ for a certain constant $\kappa>0$.
 Hence, $\PP(T_n\in\mE_n)=|E_n|8^{-n}\geq \kappa n^{-4}$. Thus, $\PP(W_n\notin [(1-\eps)n/32,(1+\eps)n/32])$ is exponentially small.
This concludes the proof that the number $W_n$ of partly reducible white faces in $R_n$ is strongly concentrated around $n/32$. We can prove similarly that the number of fully reducible white faces is strongly concentrated around $3n/32$. The cases of black faces can also be treated in the same way. This concludes the proof of Proposition~\ref{prop:reduc}.


\section{Conclusion and open questions}
We have shown that the definition of Schnyder decompositions (and their incarnations as orientations or corner-labellings) originally considered for simple triangulations and simple quadrangulations can actually be generalized to $d$-angulations of girth $d$ for all $d\geq 3$. In the case $d=3$, Schnyder woods can be defined on the larger, and self-dual, class of 3-connected plane graphs~\cite{F01},
with the nice feature that a Schnyder wood admits a dual \emph{in the same class}.
We have defined dual Schnyder decompositions for $d\geq 3$, but these dual (regular)
 structures are not in the same class (the Schnyder decomposition is on a
$d$-angulation, whereas the dual regular decomposition is on a $d$-regular plane graph).
Therefore we ask the following question:

\begin{Question}
For $d\geq 4$, is there a self-dual class $\mathcal{C}_d$ of plane graphs containing the $d$-angulations of girth $d$ such that the definition of Schnyder decompositions can be extended to $\mathcal{C}_d$, and be stable under duality?
\end{Question}

Our second question is related to the computational aspect of $d/(d-2)$-orientations.
Our existence proof (Theorem~\ref{thm:exists}) does not yield an algorithm. However, it is known that the computation of an $\alpha$-orientation can be reduced to a flow problem, and has polynomial complexity in the number of edges (see~\cite{Felsner:lattice}). Hence, for fixed $d\geq 3$ a $d/(d-2)$-orientation can be computed in polynomial time. For $d=3$ or $d=4$ it is known that a $d/(d-2)$-orientation (and the associated Schnyder decomposition) can be computed in \emph{linear time}, hence the following question:

\begin{Question}
For $d\geq 5$, is it possible to compute a $d/(d-2)$-orientation of a $d$-angulation in time linear in the number of vertices?
\end{Question}

The next questions are about drawing algorithms. In Section~\ref{sec:drawing} we have presented, based on dual of Schnyder decompositions,  two drawing algorithms for 4-regular plane graphs of mincut $4$. 
Like Schnyder's straight-line drawing algorithm for triangulations~\cite{Schnyder:wood2}
and the more recent straight-line drawing algorithm by Barri\`ere and Huemer~\cite{Barriere-Huemer:4-Labelings-quadrangulation}
for simple quadrangulations, our procedure relies on face-counting operations.
We raise the following questions:
\begin{Question}
Is there a more general planar drawing algorithm based on Schnyder decompositions which works for any integer $d\geq 3$ (either, for $d$-angulations of girth $d$, or for $d$-regular plane graphs of mincut $d$)? 
\end{Question}

\begin{Question}
For $p\geq 2$, is there an algorithm based on Schnyder decompositions for placing the non-root vertices of a $2p$-regular graph of mincut $2p$ on the lattice $\mathbb{Z}^p$, and drawing the non-root edges by a sequence of segments in the direction of the axes with at most one bend per edge? What about $p-1$ bends per edges?
\end{Question}

\bibliographystyle{plain}

\bibliography{biblio-Schnyder}

\end{document}